\documentclass[hidelinks,onefignum,onetabnum]{siamart220329}

\usepackage{lipsum}
\usepackage{amsfonts}
\usepackage{graphicx}
\usepackage{epstopdf}
\usepackage{algorithmic}
\ifpdf
  \DeclareGraphicsExtensions{.eps,.pdf,.png,.jpg}
\else
  \DeclareGraphicsExtensions{.eps}
\fi

\newsiamremark{remark}{Remark}
\newsiamremark{hypothesis}{Hypothesis}
\crefname{hypothesis}{Hypothesis}{Hypotheses}
\newsiamthm{claim}{Claim}
\newsiamthm{conj}{Conjecture}
\newsiamremark{Ex}{Example}

\headers{Fully-discrete entropy stable DG schemes for MHD}{}

\title{Limiter-based fully-discrete entropy stable explicit DG schemes for ideal MHD equations}

\author{Yuchang Liu\thanks{School of Mathematical Sciences,
         University of Science and Technology of China,
         Hefei, Anhui 230026, P.R. China.
  (\email{lissandra@mail.ustc.edu.cn}).}
\and Yan Jiang\thanks{ School of Mathematical Sciences,
         University of Science and Technology of China, Hefei, 
         Anhui 230026, P.R. China.  
         (\email{jiangy@ustc.edu.cn}).  
         Research supported by NSFC grant 12671485, 12271499.}
\and Zheng Sun\thanks{Department of Mathematics, The University of Alabama,
		Tuscaloosa, AL 35487, USA. 
(\email{zsun30@ua.edu}).} 
}

\usepackage{amsopn}

\ifpdf
\hypersetup{
	pdftitle={Limiter-based fully-discrete entropy stable explicit DG schemes for ideal MHD equations},
    pdfauthor={}
}
\fi

\allowdisplaybreaks

\newtheorem{assumption}{Assumption}[section]
\usepackage{subfigure}
\usepackage{mathrsfs}
\usepackage{enumitem}

\begin{document}
	
	\maketitle

	\begin{abstract}
		We propose a class of high-order fully-discrete entropy stable (ES) explicit discontinuous Galerkin (DG) solvers for the compressible ideal magnetohydrodynamics (MHD) equations. Our main theoretical contribution is the introduction of a novel generalized-path-decomposition framework for MHD equations in Godunov's symmetric form. By innovatively interpreting the interior volume integral of the non-conservative source term as a path integral along a generalized path constructed by the solution polynomial, we establish the weak cell entropy inequality for the fully-discrete DG schemes. This overarching framework also accommodates other existing DG solvers based on the symmetric form. Combined with a carefully designed ES limiter, the proposed scheme satisfies the genuine fully-discrete cell entropy inequality. With this property, a Lax--Wendroff-type theorem can be obtained to show that the solution limit satisfies the entropy condition. Finally, the scheme is naturally compatible with the locally divergence-free space. Extensive numerical experiments demonstrate the scheme's low numerical dissipation and strong robustness.
	\end{abstract}
	

	
	\section{Introduction}
	\label{sec1}

Ideal magnetohydrodynamics (MHD) is a core theoretical model for understanding the dynamic behavior of plasmas under magnetic fields. Its applications range from space weather prediction and astrophysical jet simulation to frontier scientific fields such as controlled thermonuclear fusion. Mathematically, the MHD equations form a complex system of nonlinear hyperbolic conservation laws. Numerically, the unavoidable emergence of singular structures like shock waves makes the design of high‑fidelity and robust numerical algorithms a persistent difficulty.

The discontinuous Galerkin (DG) method \cite{shu2009discontinuous} has emerged as a prominent high-order numerical technique over the past few decades, characterized by its exceptional geometric flexibility, compact computational stencil, and high scalability for parallel computing. Consequently, the application of DG methods to solve MHD equations has gained significant traction in the field. However, a persistent challenge remains in preserving the intrinsic physical structures of the MHD system.

A primary issue in this context is preserving discrete entropy stability, which ensures consistency with the second law of thermodynamics. However, constructing an entropy stable (ES) DG scheme for the MHD equations poses several non-trivial difficulties. For general hyperbolic conservation laws, high-order methods typically do not satisfy the ES property by default. To address this, two primary strategies have emerged within the DG framework: the first involves utilizing summation-by-parts (SBP) operators under the discontinuous Galerkin spectral element method (DGSEM) framework \cite{chen2017entropy, liu2024entropy, liu2025structure}, while the second relies on adding artificial dissipation terms to balance entropy production \cite{abgrall2018general, gaburro2023high,chan2025artificial}. While these approaches are well-established for standard conservation laws, the MHD system presents a unique theoretical hurdle. As Godunov \cite{godunov1972symmetric} demonstrated, the MHD system is not symmetrizable in its pure conservative form; instead, a non-conservative source term (also known as the Godunov--Powell source term) must be introduced to restore symmetrizability. This requirement makes the entropy analysis for MHD significantly more complex than that of general systems. To tackle this challenge, Chandrashekar \textit{et al.} \cite{chandrashekar2016entropy} proposed a finite volume (FV) ES framework based on the symmetric form of the MHD equations, which was later extended to DGSEM by Liu \textit{et al.} \cite{liu2018entropy}. 

However, the aforementioned ES techniques are primarily developed at the semi-discrete level. 
Fully-discrete entropy stability is often achieved through implicit time discretizations
\cite{chen2020review}, whereas the construction of high-order explicit schemes with rigorous
fully-discrete entropy guarantees remains much less developed. Recently, in
\cite{liu2026limiter}, building on the notions of numerical entropy flux
\cite{kivva2022entropy,kivva2024entropy} and the weak entropy inequality for cell averages
\cite{carlier2023invariant}, we developed a limiter-based explicit DG framework for enforcing
fully-discrete entropy stability; see also \cite{wu2026epounifiedframeworkentropy}. Compared
with the global relaxation Runge--Kutta approach in \cite{ranocha2020relaxation}, our method
directly enforces a local cell entropy inequality through an explicit limiter and can
simultaneously accommodate multiple entropy inequalities. The resulting procedure is simple
and non-intrusive: it leaves the underlying spatial discretization unchanged and requires only
the application, after each time-advancement step, of a Zhang--Shu-type scaling limiter based
on a computable entropy upper bound.

In this work, building upon the framework established in \cite{liu2026limiter}, we aim to develop a fully-discrete ES explicit DG scheme for MHD equations. The primary difficulty lies in the influence of Godunov's source term. While the strategy in \cite{liu2026limiter} inspired by \cite{zhang2010positivity} allows for decomposing the update into three-point schemes to exploit the convexity of the entropy function, this approach encounters significant hurdles when applied to standard discretizations of the source term \cite{liu2018entropy, wu2018provably, liu2025structureMHD}. Specifically, several residual terms appear that prevent a rigorous proof of the weak entropy inequality for the MHD system. To address this, we innovatively employ the generalized-path-decomposition framework to analyze the symmetric form of the MHD equations, treating it as a genuinely non-conservative system \cite{pares2006numerical}. \textit{A key insight of our work is that the intra-cell integral of the source term can be naturally interpreted as a generalized path integral along a polynomial path.}
The proposed DG formulation remains path-conservative with respect to the prescribed path defining the non-conservative product (Remark \ref{rmk:path}), while the generalized paths are introduced only as an analytical tool for the entropy analysis. Within this generalized-path-decomposition framework, we rigorously establish the weak cell entropy inequality (Theorem \ref{thm:ESlike1D}). By incorporating the limiter introduced in \cite{liu2026limiter}, we further obtain a fully-discrete ES DG scheme satisfying the genuine cell entropy inequality (Theorem \ref{thm:genuine-ES}). With this property, a Lax--Wendroff-type theorem can be obtained to show that the limit of the numerical solutions satisfies the entropy condition (Theorem \ref{thm:LW}). Moreover, the entropy analysis can be adapted to other classical DG discretizations of Godunov's symmetric form with only minor modifications (Remark \ref{rmk:other}). Finally, the proposed scheme is naturally compatible with locally divergence-free (LDF) magnetic fields (Section \ref{sec:LDF}).  Without this property,
the scheme may produce nonphysical solutions or even lead to numerical instability
and eventual breakdown. A comprehensive suite of numerical tests demonstrates that the proposed scheme achieves optimal convergence rates, low numerical dissipation, and strong robustness while respecting the underlying physical constraints.

The main contributions of this paper can be summarized as follows:
\begin{itemize}[leftmargin=*]
        \item We construct a generalized-path-decomposition framework for ideal MHD equations. The proposed method formally preserves the underlying path defining the non-conservative product in Godunov's form.  
        \item The proposed method preserves the fully-discrete cell entropy inequality. In particular, along the way we prove that 
        \begin{itemize}[leftmargin=*]
            \item A class of first-order schemes satisfies the fully-discrete cell entropy inequality. 
            \item The forward-Euler DG method satisfies the weak cell entropy inequality (also referred to as ``entropy-stable-like" property in \cite{liu2026limiter}). 
        \end{itemize}
        \item A Lax--Wendroff-type theorem can be obtained to show that the limit of numerical solutions satisfies the entropy inequality.
        \item The proposed method naturally accommodates the LDF framework. 
        \end{itemize}

The remainder of this paper is organized as follows. In Section \ref{sec2}, we review the fundamental mathematical structure of the ideal MHD equations. Section \ref{sec3} details the derivation of the first-order three-point ES building block. Section \ref{sec4} introduces the high-order ES DG scheme in one dimension. Section \ref{sec5} extends the method to multi-dimensional cases and discusses the integration of the LDF space. Section \ref{sec6} provides comprehensive numerical examples and comparative validations. Finally, Section \ref{sec7} presents our conclusions and outlines directions for future research.

	\section{Ideal MHD equations}\label{sec2}

    \subsection{Governing equation}
   Consider the general $d$-dimensional compressible ideal MHD equations. They can be written as
\begin{equation}
\label{eq:MHD}
\frac{\partial \mathbf U}{\partial t} + \nabla\cdot\mathbf F(\mathbf U) = \mathbf{0},\quad \mathbf U(\mathbf x,0)=\mathbf U_0(\mathbf x),
\end{equation}
where 
$$
\mathbf U = \left[ \begin{array}{c}
	\rho\\
	\rho\mathbf u\\
	\mathcal E\\
	\mathbf B\\
\end{array} \right],\quad \mathbf{F}\left( \mathbf{U} \right) =\left[ \begin{array}{c}
	\rho\mathbf u\\
	\rho \mathbf u\otimes\mathbf u+p^\star\mathbf  I_d-\mathbf B\otimes\mathbf B\\
	\mathbf u(\mathcal E+p^\star)-\mathbf B(\mathbf u\cdot\mathbf B)\\
	\mathbf u\otimes\mathbf B-\mathbf B\otimes\mathbf u\\
\end{array} \right],
$$
with 
\begin{equation*}\mathcal E=\frac{p}{\gamma - 1} + \frac{1}{2}\rho\left\|\mathbf u\right\|^2+\frac{1}{2}\left\|\mathbf B\right\|^2.\end{equation*}
Here, $\rho$ is the mass density, $\rho\mathbf{u}$ is the momentum density, $\mathcal{E}$ is the total energy density, $p$ is the hydrodynamic pressure, $\mathbf{B}$ is the magnetic field, and $\gamma = 5/3$ is the adiabatic index. Additionally, $p^\star = p + \left\|\mathbf{B}\right\|^2/2$ is the total pressure, $\|\cdot\|$ denotes the Euclidean vector norm, $\mathbf I_d$ is the $d\times d$ identity matrix, and $\otimes$ denotes the tensor product. We also denote the flux function component-wise by $\mathbf F(\mathbf U)=[\mathbf F_1(\mathbf U),\dots,\mathbf F_d(\mathbf U)]$.

    \subsection{Divergence-free property}
    Taking divergence of the magnetic field equation yields
$$
\frac{\partial \left( \nabla \cdot \mathbf{B} \right)}{\partial t}=0,
$$
 indicating
 $\nabla\cdot\mathbf B(\mathbf x,t) = \nabla\cdot\mathbf B(\mathbf x,0).$
Consequently, if the divergence of the magnetic field is initially zero, it will remain zero for all time, i.e. 
\begin{equation}\label{eq:divfree0}
\nabla\cdot\mathbf B = 0.
\end{equation}
This is called the \textit{divergence-free} property.  Physically, this constraint reflects the absence of magnetic monopoles.
  
 In particular, the numerical scheme must be carefully designed to preserve this divergence-free property in the discrete sense, which is critical for preventing nonphysical artifacts and numerical instabilities in MHD simulations.

    \subsection{Entropy structure and symmetrizable form}
	For a general conservation law of the form \eqref{eq:MHD}, the entropy pair is defined as follows.
	
	\begin{definition}[Entropy pair]\label{def:entropy_pair}
		A convex function $\mathcal{U}(\mathbf{U})$ is  an entropy function for system \eqref{eq:MHD} if there exist entropy fluxes $\boldsymbol{\mathcal F}(\mathbf U)=[\mathcal F_1(\mathbf U),\cdots,\mathcal F_d(\mathbf U)]$ such that 
		\begin{equation*}
			\mathcal F_i'(\mathbf U)=\mathcal U'(\mathbf U)\mathbf F_i'(\mathbf U),\quad i=1,\cdots,d.
		\end{equation*}
		We call $(\mathcal U,\boldsymbol{\mathcal F})$ an entropy pair.
	\end{definition}

	Let $\mathbf V=\mathcal U'(\mathbf U)^T$ denote the \textit{entropy variable}. If $\mathcal U$ is strictly convex, then the mapping $\mathbf U\to \mathbf V$ is one-to-one. We can rewrite \eqref{eq:MHD} in terms of entropy variables:
	\begin{equation}\label{eq:MHD-V}
		\mathbf U'(\mathbf V)\frac{\partial \mathbf V}{\partial t}+\sum\limits_{i=1}^d\mathbf F_i'(\mathbf U)\mathbf U'(\mathbf V)\frac{\partial\mathbf V}{\partial x_i} = \mathbf 0.
	\end{equation}
	Due to the strict convexity of $\mathcal U$, $\mathbf U'(\mathbf V)=(\mathcal U''(\mathbf U))^{-1}$ is positive definite. If the matrices $\mathbf F_i'(\mathbf U)\mathbf U'(\mathbf V)$ are symmetric, then \eqref{eq:MHD-V} is called a \textit{symmetrization} of \eqref{eq:MHD}, and \eqref{eq:MHD} is said to be symmetrizable. An important result in \cite{godlewski2013numerical} shows that the symmetrizability of \eqref{eq:MHD} is equivalent to the existence of an entropy function.

	If a system of conservation laws admits an entropy pair, a weak solution is called the \textit{entropy solution} if it satisfies the \textit{entropy condition}
\begin{equation}\label{eq:entropy_condition}
\frac{\partial \mathcal U\left( \mathbf{U} \right)}{\partial t}+\nabla\cdot\boldsymbol{\mathcal F}(\mathbf U)\le 0
\end{equation}
in the weak sense. Formally, this means that 
\begin{equation}\label{eq:ESweak}
\int_{\mathbb R^+}\int_{\mathbb R^d}(\mathcal U(\mathbf U)\phi_t+\boldsymbol{\mathcal F}(\mathbf U)\cdot\nabla\phi )\mathrm d\mathbf x\mathrm dt\ge -\int_{\mathbb R^d}\mathcal U(\mathbf U_0(\mathbf x))\phi(\mathbf x,0)\mathrm d\mathbf x
\end{equation}
for any test function $\phi\in C_0^\infty(\mathbb R^d\times\mathbb R^+)$ with $\phi\ge 0$. In particular, the inequality \eqref{eq:entropy_condition} becomes an equality for smooth solutions.

Integrating \eqref{eq:entropy_condition} in space yields
    \begin{equation}\label{eq:es_semi}
        \frac{\mathrm d}{\mathrm dt}\int_{\mathbb R^d}\mathcal U(\mathbf U)\mathrm d\mathbf x\le 0.
    \end{equation}
Further integrating \eqref{eq:entropy_condition} in time
gives 	\begin{equation}\label{eq:ESglobal_fully}
		\int_{\mathbb R^d}\mathcal U(\mathbf U(\mathbf x,t+\Delta t))\mathrm d\mathbf x\le \int_{\mathbb R^d}\mathcal U(\mathbf U(\mathbf x,t))\mathrm d\mathbf x\, \quad \forall t,\Delta t>0.
	\end{equation}
A numerical method preserving  \eqref{eq:es_semi} or \eqref{eq:ESglobal_fully} is called \emph{entropy stable}. 
    
    Let $s = \ln (p \rho^{-\gamma})$. For ideal MHD equations, we can show that the quantities \begin{equation}\label{eq:entropy_pair} 
    \mathcal U=-\frac{\rho s}{\gamma - 1}, \qquad \boldsymbol{\mathcal F}=-\frac{\rho s\mathbf u }{\gamma - 1} 
    \end{equation}
    satisfy
    \begin{equation}\label{eq:d_rhos}
    \frac{\partial \rho s}{\partial t}+\nabla\cdot(\rho s\mathbf u)+\left( \gamma -1 \right) \frac{\rho \left( \mathbf{u}\cdot \mathbf{B} \right)}{p}\left( \nabla \cdot \mathbf{B} \right) =0
    \end{equation}
    for smooth solutions. By the divergence-free property $\nabla\cdot\mathbf B=0$, the last term in \eqref{eq:d_rhos} vanishes, thus it is natural to expect that the above $(\mathcal U,\boldsymbol{\mathcal F})$ is an entropy pair of \eqref{eq:MHD}. However, it does not satisfy Definition \ref{def:entropy_pair}. Instead, the following relation holds:
    \begin{equation}\label{eq:phi}
    \mathcal F_i'(\mathbf U)=\mathcal U'(\mathbf U)\mathbf F_i'(\mathbf U) + \phi(\mathbf V)B_{x_i}'. 
    \end{equation}
    Here, $\phi(\mathbf V)=2\beta(\mathbf u\cdot\mathbf B)$ and $\beta=\rho/2p$. Note that the subscript $x_i$ in $B_{x_i}$ denotes the spatial direction, not a partial derivative. The gradient $B'_{x_i} = \partial_{\mathbf{U}} B_{x_i}$ is a constant row vector matching the length of $\mathbf{U}$; its elements are 1 at the position corresponding to $B_{x_i}$ and 0 elsewhere. For example, if $\mathbf{U} = [\rho, \rho u_x, \mathcal{E}, B_x]^T$, then $B'_x = [0, 0, 0, 1]$.
    
    To address this, Godunov suggested the following modified ``symmetric" form
    \begin{equation}\label{eq:MHDsym}
    \frac{\partial \mathbf U}{\partial t}+\nabla\cdot\mathbf F(\mathbf U)=-\mathbf S(\mathbf U)(\nabla\cdot\mathbf B),\quad \mathbf S(\mathbf U)=\phi'(\mathbf V)^T.
    \end{equation}
    In particular, $\phi(\mathbf V)$ is homogeneous of degree one, i.e.
    \begin{equation*}
    \mathbf S(\mathbf U)\cdot\mathbf V=\phi'(\mathbf V)\mathbf V=\phi(\mathbf V).
    \end{equation*} 
    For this modified form, $\mathbf U\to \mathbf V$ is a symmetrization, hence $(\mathcal U,\boldsymbol{\mathcal F})$ in \eqref{eq:entropy_pair} is an entropy pair of \eqref{eq:MHDsym}.

    Following \cite{liu2026limiter}, we will develop a high-order fully-discrete DG method that solves the symmetric form \eqref{eq:MHDsym} and satisfies \eqref{eq:ESglobal_fully}, while locally enforcing the divergence-free condition \eqref{eq:divfree0}.

    \subsection{Non-conservative product and paths} 

    Although \eqref{eq:MHDsym} is equivalent to \eqref{eq:MHD} for smooth solutions, the non-conservative product $-\mathbf S(\mathbf U)(\nabla\cdot\mathbf B)$ must be carefully defined at discontinuities, since integration by parts cannot shift all spatial derivatives onto the test functions in the standard weak form. To resolve this issue, consider a general non-conservative system
\begin{equation}\label{eq:nc}
\mathbf U_t + \mathbf A(\mathbf U)\mathbf U_x=\mathbf 0.
\end{equation}
At a discontinuity $x_\star$, the theory in \cite{dal1995definition} suggests considering a smooth regularization of $\mathbf U$ by $\mathbf{U}^\varepsilon$. It connects the left state $\mathbf U^-$ and the right state $\mathbf U^+$ by $\mathbf{U}^\varepsilon = \Psi((x-x_\star+\varepsilon)/(2\varepsilon);\mathbf{U}^-,\mathbf{U}^+)$ if $|x-x_\star|<\varepsilon$. Here $\Psi$ is a \emph{path} defined as follows.

\begin{definition}\label{def:pathfun}
A path is a Lipschitz continuous function $\Psi:[0,1]\times \mathbb R^p\times\mathbb R^p\to \mathbb R^p$ that satisfies 
\begin{equation}\label{eq:pathdef}
\Psi(0;\mathbf U^-,\mathbf U^+)=\mathbf U^-,\quad \Psi(1;\mathbf U^-,\mathbf U^+)=\mathbf U^+,\quad  \Psi(\xi;\mathbf U,\mathbf U)=\mathbf U.
\end{equation}
When no confusion arises, we denote it simply by $\Psi(\xi)$.
\end{definition}

Taking $\varepsilon\to 0$, $\mathbf A (\mathbf U) \mathbf U_x$ is defined as a Borel measure such that
    $$
    [\mathbf A (\mathbf U) \mathbf U_x]_\Psi=\mathbf A (\mathbf U) \mathbf U_x \mathrm{d} x + \sum_{x_\star} \left(\int_0^1 \mathbf A(\Psi(\xi))\frac{\partial \Psi}{\partial \xi}(\xi)\mathrm d\xi\right)\delta(x_\star),
    $$
    where the summation is taken over all discontinuous points $x_\star$, the left and right states in $\Psi$ vary with $x_\star$, and $\delta$ is the Dirac measure. 
    With the non-conservative product well-defined, the weak solution to \eqref{eq:nc} can be defined accordingly. 

In this framework, the definition of weak solutions depends heavily on the choice of paths, whose selection is \emph{a priori} arbitrary. To obtain the physically relevant solution, the path must be chosen based on physical considerations---most rigorously, by studying the vanishing viscosity limit \cite{castro2013entropy}: different underlying dissipation mechanisms dictate different shock structures that prescribe the admissible paths.

In the context of the ideal MHD equations, the effect of paths is usually underemphasized. In this work, we will derive the numerical scheme for any given path, but only consider the linear path in the numerical tests:
\begin{equation}\label{eq:linpath}
    \Psi(\xi)=\mathbf U^-+\xi (\mathbf U^+-\mathbf U^-).
\end{equation}

	\section{First-order building block in one dimension}
	\label{sec3}

    For simplicity, we first consider the 1D form of \eqref{eq:MHDsym}
    \begin{equation}\label{eq:MHD1D}
    \frac{\partial \mathbf U}{\partial t}+\frac{\partial \mathbf F(\mathbf U)}{\partial x}=-\mathbf S(\mathbf U)\frac{\partial B_x}{\partial x}.
    \end{equation}
    We assume $\mathbf U \in \mathbb R^p$, and $p$ is not necessarily $4$, as \eqref{eq:MHD1D} may represent the general multidimensional system \eqref{eq:MHDsym} restricted to a flow field varying only in the $x$-direction.  Moreover, although it is implied by the divergence-free condition in the 1D case, we do not assume $B_x$ to be constant to preserve the general structure of the scheme.
    
    In this section, we study first-order, three-point ES schemes for \eqref{eq:MHD1D}, which serve as the building blocks for constructing high-order ES schemes \cite{liu2026limiter}. Because the entropy pair of the symmetric form of the MHD equations differs from that of standard conservation laws---with special consideration required for approximating the non-conservative source---the analysis of the first-order scheme is fundamentally different and significantly more challenging.

    \subsection{ Path-conservative scheme} 
    
    Assume the spatial domain $\mathbb R$ is divided into uniform cells $I_i=[x_{i-1/2},x_{i+1/2}]$, and the cell average on $I_i$ is denoted by $\bar{\mathbf U}_i$. We treat \eqref{eq:MHD1D} as a non-conservative system. Consider the following three-point scheme:
    \begin{equation}\label{eq:first}
    \begin{aligned}
    \bar{\mathbf U}_i^{n+1}=&\,\bar{\mathbf U}_i^n-\lambda\left( \hat{\mathbf F}(\bar{\mathbf U}_i^n,\bar{\mathbf U}_{i+1}^n)-\hat{\mathbf F}(\bar{\mathbf U}_{i-1}^n,\bar{\mathbf U}_{i}^n) \right)
    \\ &-\lambda\left( \mathcal D^-(\bar{\mathbf U}_i^n,\bar{\mathbf U}_{i+1}^n)+\mathcal D^+(\bar{\mathbf U}_{i-1}^n,\bar{\mathbf U}_i^n) \right).
    \end{aligned}
    \end{equation}
    Here, $\hat{\mathbf F}$ represents the numerical flux. In this work, we use the HLL flux
    \begin{equation}\label{eq:HLL}
    \hat{\mathbf F}^{\mathrm{HLL}}(\mathbf U^-,\mathbf U^+)=\frac{\mathcal S_R\mathbf F(\mathbf U^-)-\mathcal S_L\mathbf F(\mathbf U^+)+\mathcal S_R\mathcal S_L(\mathbf U^+-\mathbf U^-)}{\mathcal S_R-\mathcal S_L},
    \end{equation}
    where $\mathcal S_R\ge 0,\ \mathcal S_L\le 0$ are estimates of maximum and minimum signed wave speed. The standard setup is
    \begin{equation}\label{eq:SRSLstd}
    \begin{aligned}\mathcal S_R^{\mathrm{std}}&=\max\{\sigma_{\max}(\mathbf F'(\mathbf U^-)),\sigma_{\max}(\mathbf F'(\mathbf U^+)),0\},
    \\ \quad \mathcal S_L^{\mathrm{std}}&=\min\{\sigma_{\min}(\mathbf F'(\mathbf U^-)),\sigma_{\min}(\mathbf F'(\mathbf U^+)),0\},
    \end{aligned}
    \end{equation}
    and $\sigma_{\min}(\mathbf A),\sigma_{\max}(\mathbf A)$ are the minimum and maximum eigenvalues of $\mathbf A$, respectively. 
    However, the standard setup \eqref{eq:SRSLstd} may not ensure entropy stability, and the wave speed in \eqref{eq:HLL} will be determined later. 
    
    The key to the scheme \eqref{eq:first} is the introduction of the $\mathcal D^\pm$ terms to approximate Godunov's source term $\mathbf S(\mathbf U)(\nabla\cdot\mathbf B)$. Following the path-conservative framework \cite{pares2006numerical}, we impose specific requirements on these operators, as detailed below.
    
    \begin{definition}[Path-conservative scheme] \label{def:path}
    The scheme \eqref{eq:first} is called path-conservative if the operator $\mathcal D^\pm$ satisfies
    \begin{equation}\label{eq:pathcons}
    \mathcal D^+({\mathbf U}^-,{\mathbf U}^+)+\mathcal D^-(\mathbf U^-,\mathbf U^+)=\int_0^1\mathbf S(\Psi(\xi;\mathbf U^-,\mathbf U^+))B_x'\frac{\partial\Psi}{\partial\xi}(\xi;\mathbf U^-,\mathbf U^+)\mathrm d\xi.
    \end{equation}
    \end{definition}
    
    To achieve path conservation prescribed by \eqref{eq:pathcons}, we define $\mathcal D^\pm$ in \eqref{eq:first} as
    $$ \mathcal D^\pm(\mathbf U^-,\mathbf U^+)=\mathcal R^\pm(\mathbf U^-,\mathbf U^+)\int_0^1\mathbf S(\Psi(\xi))B_x'\Psi'(\xi)\mathrm d\xi.$$
    Here, the ratio $\mathcal R^\pm$ is defined by
    \begin{equation}\label{eq:R} \mathcal R^+(\mathbf U^-,\mathbf U^+)=\frac{\mathcal S_R}{\mathcal S_R-\mathcal S_L},\quad \mathcal R^-(\mathbf U^-,\mathbf U^+)=\frac{-\mathcal S_L}{\mathcal S_R-\mathcal S_L}, 
    \end{equation}
    which satisfies $\mathcal R^+(\mathbf U^-,\mathbf U^+)+\mathcal R^-(\mathbf U^-,\mathbf U^+)=1$. This property ensures that the definition meets the requirement of a path-conservative scheme. We denote $\mathcal R^\pm_{i,i+1}=\mathcal R^\pm(\bar{\mathbf U}_i,\bar{\mathbf U}_{i+1})$. In particular, for LF flux, we have $\mathcal S_R=-\mathcal S_L$, in which case $\mathcal R^+=\mathcal R^-=1/2$; hence the analysis also applies to the LF flux. 

    \subsection{A generalized scheme} 
    Note that Definition \ref{def:path} requires the operator $\mathcal D^\pm$ to use the path $\Psi$ defining the PDE at each cell interface to connect $\bar{\mathbf U}_i^n$ and $\bar{\mathbf U}_{i+1}^n$. However, to establish the entropy stability of high-order schemes, we need to consider a more general case, where $\bar{\mathbf U}_i^n$ and $\bar{\mathbf U}_{i+1}^n$ can be connected by \emph{generalized paths} varying at different cell interfaces.

    \begin{definition}
    For a given set of $\{\bar{\mathbf U}_i^n\}_{i\in\mathbb Z}$, we call a Lipschitz continuous function $\Phi_{i,i+1}:[0,1]\to\mathbb R^p$ a generalized path connecting $\bar{\mathbf U}_i^n$ and $\bar {\mathbf U}_{i+1}^n$ if
    $$ \Phi_{i,i+1}(0)=\bar{\mathbf U}_i^n,\quad \Phi_{i,i+1}(1)=\bar{\mathbf U}_{i+1}^n. $$
    \end{definition}
    
    From the definition, we can see that $\Phi_{i,i+1}$ may vary with $i$, and does not satisfy the third condition in \eqref{eq:pathdef}. Utilizing the generalized path, we modify the scheme \eqref{eq:first} into a more general scheme
    \begin{equation}\label{eq:first2}
    \begin{aligned}
    \bar{\mathbf U}_i^{n+1}=&\,\bar{\mathbf U}_i^n-\lambda\left( \hat{\mathbf F}(\bar{\mathbf U}_i^n,\bar{\mathbf U}_{i+1}^n)-\hat{\mathbf F}(\bar{\mathbf U}_{i-1}^n,\bar{\mathbf U}_{i}^n) \right)
    \\ &-\lambda\left( \mathcal A^-_{i,i+1}(\bar{\mathbf U}_i^n,\bar{\mathbf U}_{i+1}^n)+\mathcal A^+_{i-1,i}(\bar{\mathbf U}_{i-1}^n,\bar{\mathbf U}_i^n) \right),
    \end{aligned}
    \end{equation}
    where
    \begin{equation}\label{eq:A-def}
        \mathcal A^\pm_{i,i+1}(\bar{\mathbf U}_i^n,\bar{\mathbf U}_{i+1}^n)=\mathcal R^\pm_{i,i+1}\int_0^1\mathbf S(\Phi_{i,i+1}(\xi))B_x'\Phi_{i,i+1}'(\xi)\mathrm d\xi 
    \end{equation}
    also depends on $i$. When $\Phi_{i,i+1}(\xi)=\Psi(\xi;\bar{\mathbf U}_i,\bar{\mathbf U}_{i+1})$ for all $i$, \eqref{eq:first2} reduces to \eqref{eq:first}. 
    
    Despite the introduction of different paths, it should be emphasized that a standard path-conservative scheme still employs the same binary operators $\mathcal D^\pm$ at all interfaces $i$. The generalized paths are introduced solely as a theoretical tool to prove the properties of the high-order scheme; thus, \eqref{eq:first2} is never actually implemented in practice for solving the MHD equations. 

    \subsection{A cell entropy inequality}  Following the framework in \cite{kivva2022entropy, kivva2024entropy, liu2026limiter}, to establish the fully-discrete entropy stability, the definition of numerical entropy flux is crucial. For MHD equations, we suggest the following HLL numerical entropy flux 
    \begin{equation}\label{eq:entropyHLL}
    \hat{\mathcal F}(\mathbf U^-,\mathbf U^+)=\frac{\mathcal S_R\mathcal F(\mathbf U^-)-\mathcal S_L\mathcal F(\mathbf U^+)+\mathcal S_R\mathcal S_L(\mathcal U(\mathbf U^+)-\mathcal U(\mathbf U^-))}{\mathcal S_R-\mathcal S_L}.
    \end{equation}
    It can be verified that \begin{equation*}
    \frac{\partial \hat{\mathcal F}}{\partial\mathbf U^{\pm}}(\mathbf U^-,\mathbf U^+)=\mathcal U'(\mathbf U^\pm)\frac{\partial \hat{\mathbf F}}{\partial\mathbf U^{\pm}}(\mathbf U^-,\mathbf U^+)+\mathcal {R^\mp}\phi(\mathbf V^\pm) B_x'.
    \end{equation*}
    Recall that $\phi$ and $B_x'$ were defined in \eqref{eq:phi}.  Moreover, note that we treat $\mathcal S_R,\mathcal S_L$ as fixed constants when taking derivatives. 
    By using the numerical entropy flux \eqref{eq:entropyHLL}, the entropy stability can be established as follows. The proof follows the idea in \cite{kivva2022entropy, liu2026limiter}. For notational convenience, the partial derivatives of $\hat{\mathbf F}$ are denoted by $\hat{\mathbf F}_1$ and $\hat{\mathbf F}_2$, while those of $\hat{\mathcal F}$ are denoted by $\hat{\mathcal F}_1$ and $\hat{\mathcal F}_2$. 
    
    \begin{theorem}\label{thm:ES1st} Assume $\bar{\mathbf U}_i^n\ne \bar{\mathbf U}_{i\pm 1}^n$. Let $L(\mathbf a,\mathbf b)=\{s\mathbf a+(1-s)\mathbf b:s\in [0,1]\}$ be the line segment connecting $\mathbf a$ and $\mathbf b$. Denote $\mathbf d^\pm=\bar{\mathbf U}_{i\pm 1}^n-\bar{\mathbf U}_i^n$, $\mathbf U_\xi^\pm =\bar{\mathbf U}_i^n+\xi\mathbf d^\pm $. If the wave speeds  satisfy
    \begin{equation}\label{eq:speed}
    \begin{aligned}
     \mathcal S_R(\bar{\mathbf U}_i,\bar{\mathbf U}_{i+1})&\ge \frac{\displaystyle \int_0^1\int_0^\xi (\mathbf d^+)^T \mathcal U''(\mathbf U_\zeta^+) {\mathbf F}'(\mathbf U_\xi^+) \mathbf d^+ \mathrm d\zeta\mathrm d\xi+\Delta\mathcal A^+}{\displaystyle \int_0^1\int_0^\xi (\mathbf d^+)^T \mathcal U''(\mathbf U_\zeta^+)  \mathbf d^+ \mathrm d\zeta\mathrm d\xi},
     \\ 
     \mathcal S_L(\bar{\mathbf U}_{i-1},\bar{\mathbf U}_i)&\le  \frac{\displaystyle \int_0^1\int_0^\xi (\mathbf d^-)^T \mathcal U''(\mathbf U_\zeta^-) {\mathbf F}'(\mathbf U_\xi^-) \mathbf d^- \mathrm d\zeta\mathrm d\xi+\Delta \mathcal A^-}{\displaystyle \int_0^1\int_0^\xi (\mathbf d^-)^T \mathcal U''(\mathbf U_\zeta^-)  \mathbf d^- \mathrm d\zeta\mathrm d\xi},
    \end{aligned}
    \end{equation}
    then the scheme \eqref{eq:first2} is entropy stable in the sense of
    \begin{equation*}
    \mathcal U(\bar{\mathbf U}_i^{n+1})\le \mathcal U(\bar{\mathbf U}_i^n)-\lambda\left(\hat{\mathcal F}(\bar{\mathbf U}_i^n,\bar{\mathbf U}_{i+1}^n)-\hat{\mathcal F}(\bar{\mathbf U}_{i-1}^n,\bar{\mathbf U}_{i}^n)\right)
    \end{equation*}
    under the CFL condition 
    \begin{equation}\label{eq:CFL1st}
    0\le  \lambda\le \frac{2(T_2-T_1)}{\left\|T_3\right\|^2}\left(\max\limits_{\xi\in L(\bar{\mathbf U}_i^n,\bar{\mathbf U}_i^{n+1})}\sigma_{\max}(\mathcal{U} ''(\xi))\right)^{-1}.
    \end{equation}
    Here, omitting the superscript $n$, the notations are respectively
    \begin{align*}
        T_1=&\,\hat{\mathcal F}(\bar{\mathbf U}_i,\bar{\mathbf U}_{i+1})-\hat{\mathcal F}(\bar{\mathbf U}_i,\bar{\mathbf U}_i)
    \\ &-\mathcal U'(\bar{\mathbf U}_i)\left( \hat{\mathbf F}(\bar{\mathbf U}_i,\bar{\mathbf U}_{i+1}) - \hat{\mathbf F}(\bar{\mathbf U}_i,\bar{\mathbf U}_i) + \mathcal A_{i,i+1}^-(\bar{\mathbf U}_{i}^n,\bar{\mathbf U}_{i+1}^n)\right),
    \\T_2=&\,\hat{\mathcal F}(\bar{\mathbf U}_{i-1},\bar{\mathbf U}_{i})-\hat{\mathcal F}(\bar{\mathbf U}_i,\bar{\mathbf U}_i)
    \\ &-\mathcal U'(\bar{\mathbf U}_i)\left( \hat{\mathbf F}(\bar{\mathbf U}_{i-1},\bar{\mathbf U}_{i}) - \hat{\mathbf F}(\bar{\mathbf U}_i,\bar{\mathbf U}_i) -\mathcal A_{i-1,i}^+(\bar{\mathbf U}_{i-1}^n,\bar{\mathbf U}_i^n) \right), 
    \\ T_3=&\,\hat{\mathbf F}(\bar{\mathbf U}_i,\bar{\mathbf U}_{i+1}) - \hat{\mathbf F}(\bar{\mathbf U}_{i-1},\bar{\mathbf U}_i) + \mathcal A_{i,i+1}^-(\bar{\mathbf U}_i^n,\bar{\mathbf U}_{i+1}^n)+\mathcal A_{i-1,i}^+(\bar{\mathbf U}_{i-1}^n,\bar{\mathbf U}_i^n) ,
    \\ \Delta \mathcal A^+=&\,\int_0^1\mathcal U'(\mathbf U^+_\xi)\mathbf S(\mathbf U_\xi^+)B_x'\mathbf d^+-\,\mathcal U'(\bar{\mathbf U}_i)\mathbf S(\Phi_{i,i+1}(\xi))B_x'\Phi_{i,i+1}'(\xi)\mathrm d\xi,
    \\ \Delta \mathcal A^-=&\,\int_0^1\mathcal U'(\mathbf U^-_\xi)\mathbf S(\mathbf U_\xi^-)B_x'\mathbf d^-+\,\mathcal U'(\bar{\mathbf U}_i)\mathbf S(\Phi_{i-1,i}(\xi))B_x'\Phi_{i-1,i}'(\xi)\mathrm d\xi.
    \end{align*}
    If $T_3 = 0$, the corresponding CFL restriction is understood to be vacuous.
    \end{theorem}
    \begin{proof}
    By utilizing \eqref{eq:first2}, we have \begin{align*}
    \mathcal{U} &\left( \bar{\mathbf{U}}_{i}^{n+1} \right) -\mathcal{U} \left( \bar{\mathbf{U}}_{i}^{n} \right) +\lambda \left( \hat{\mathcal{F}}\left( \bar{\mathbf{U}}_{i}^{n},\bar{\mathbf{U}}_{i+1}^{n} \right) -\hat{\mathcal{F}}\left( \bar{\mathbf{U}}_{i-1}^{n},\bar{\mathbf{U}}_{i}^{n} \right) \right)
    \\
    =&\,\,\mathcal{U} \left( \bar{\mathbf{U}}_{i}^{n+1} \right) -\mathcal{U} \left( \bar{\mathbf{U}}_{i}^{n} \right) -\mathcal{U} '\left( \bar{\mathbf{U}}_{i}^{n} \right) \left( \bar{\mathbf{U}}_{i}^{n+1}-\bar{\mathbf{U}}_{i}^{n} \right)  \notag
    \\
    &+\lambda \left( \hat{\mathcal{F}}\left( \bar{\mathbf{U}}_{i}^{n},\bar{\mathbf{U}}_{i+1}^{n} \right) -\hat{\mathcal{F}}\left( \bar{\mathbf{U}}_{i-1}^{n},\bar{\mathbf{U}}_{i}^{n} \right) \right) \notag
    \\
    &-\lambda\, \mathcal{U} '\left( \bar{\mathbf{U}}_{i}^{n} \right) \left( \hat{\mathbf{F}}\left( \bar{\mathbf{U}}_{i}^{n},\bar{\mathbf{U}}_{i+1}^{n} \right) -\hat{\mathbf{F}}\left( \bar{\mathbf{U}}_{i-1}^{n},\bar{\mathbf{U}}_{i}^{n} \right) \right) \notag
    \\
    &-\lambda\, \mathcal{U} '\left( \bar{\mathbf{U}}_{i}^{n} \right) \left( \mathcal A_{i,i+1}^-(\bar{\mathbf U}_i^n,\bar{\mathbf U}_{i+1}^n)+\mathcal A_{i-1,i}^+(\bar{\mathbf U}_{i-1}^n,\bar{\mathbf U}_i^n)  \right) \notag
    \\
    \le &\,\frac{\lambda^2}{2}\left(\max\limits_{\xi\in L(\bar{\mathbf U}_i^n,\bar{\mathbf U}_i^{n+1})}\sigma_{\max}(\mathcal{U} ''(\xi))\right) \left\| T_3 \right\| ^2+\lambda \left( T_1-T_2 \right). \notag
    \end{align*} 
    To obtain \eqref{eq:CFL1st}, we only need to show that $T_1\le 0$ and $T_2\ge 0$. For $T_1$, omitting the superscript $n$ and utilizing \eqref{eq:A-def}, we have
    \begin{align*} T_1=&\,\int_{0}^1\hat{\mathcal F}_2(\bar{\mathbf U}_i,\mathbf U_\xi^+)\mathbf d^+\mathrm d\xi-\int_0^1\mathcal U'(\bar{\mathbf U}_i)\hat{\mathbf F}_2(\bar{\mathbf U}_i,\mathbf U_\xi^+)\mathbf d^+\mathrm d\xi
    \\&-\mathcal R^-_{i,i+1}\int_0^1\mathcal U'(\bar{\mathbf U}_i)\mathbf S(\Phi_{i,i+1}(\xi))B_x'\Phi_{i,i+1}'(\xi)\mathrm d\xi
    \\ =&\,\int_0^1\left(\mathcal U'(\mathbf U_\xi^+)-\mathcal U'(\bar{\mathbf U}_i)\right)\hat{\mathbf F}_2(\bar{\mathbf U}_i,\mathbf U_\xi^+)\mathbf d^+ 
    \\ &+\mathcal R_{i,i+1}^-\int_0^1(\mathcal U'(\mathbf U_\xi^+)\mathbf S(\mathbf U_\xi^+)B_x'\mathbf d^+-\mathcal U'(\bar{\mathbf U}_i)\mathbf S(\Phi_{i,i+1}(\xi))B_x'\Phi_{i,i+1}'(\xi))\mathrm d\xi
    \\ =&\, \mathcal R_{i,i+1}^-\int_0^1\int_0^\xi (\mathbf d^+)^T \mathcal U''(\mathbf U_\zeta^+) {\mathbf F}'(\mathbf U_\xi^+)  \mathbf d^+ \mathrm d\zeta\mathrm d\xi+\mathcal R_{i,i+1}^-\Delta \mathcal A^+
    \\ &-\mathcal R_{i,i+1}^-\mathcal S_R(\bar{\mathbf U}_i,\bar{\mathbf U}_{i+1})\int_0^1\int_0^\xi (\mathbf d^+)^T \mathcal U''(\mathbf U_\zeta^+)  \mathbf d^+ \mathrm d\zeta\mathrm d\xi.
    \end{align*}
    Note that in the first equation, the ratio $\mathcal R^-$ in the term $\hat{\mathcal F}_2(\bar {\mathbf U}_i,\mathbf U_\xi^+)$ remains $\mathcal R_{i,i+1}^-$ for all $\mathbf U_\xi^+$ with $0\le\xi\le 1$, but not $\mathcal R^-(\bar {\mathbf U}_i,\mathbf U_\xi^+)$. Note that $\mathcal U''$ is strictly positive definite. Hence, $T_1\le 0$ if 
    $$ \mathcal S_R(\bar{\mathbf U}_i,\bar{\mathbf U}_{i+1})\ge \frac{\displaystyle \int_0^1\int_0^\xi (\mathbf d^+)^T \mathcal U''(\mathbf U_\zeta^+) {\mathbf F}'(\mathbf U_\xi^+) \mathbf d^+ \mathrm d\zeta\mathrm d\xi+\Delta\mathcal A^+}{\displaystyle \int_0^1\int_0^\xi (\mathbf d^+)^T \mathcal U''(\mathbf U_\zeta^+)  \mathbf d^+ \mathrm d\zeta\mathrm d\xi}. $$
    The analysis for the $T_2$ term can proceed analogously. 
    \end{proof}
    \begin{remark}
    Although the term $\max_{\xi\in L(\bar{\mathbf U}_i^n,\bar{\mathbf U}_i^{n+1})}\sigma_{\max}(\mathcal{U} ''(\xi))$ depends on the future time level $t^{n+1}$, making condition \eqref{eq:CFL1st} implicit, it is worth noting that \eqref{eq:CFL1st} is merely a sufficient condition. In practice, a feasible approach is to verify whether the entropy inequality is strictly satisfied after computation with a standard time step. If the entropy inequality is violated, the algorithm rejects the step and recomputes with a smaller time step. This is reasonable since in \cite{liu2026limiter}, we found that \eqref{eq:CFL1st} is close to the standard CFL condition. As a special simplified case, \eqref{eq:CFL1st} will reduce to $\lambda\le 1$ for the linear equation $u_t+u_x=0$ with square entropy.
    \end{remark}

    \begin{remark}\label{remark:path}
    For linear paths, it can be verified that the bounds for $\mathcal S_R$
    and $\mathcal S_L$ in \eqref{eq:speed} remain $\mathcal O(1)$ and can
    be interpreted as weighted averages of the Jacobian of \eqref{eq:MHD1D}.
    More generally, for the polynomial generalized paths arising in the DG
    analysis, uniformly bounded wave-speed estimates can be obtained for every
    fixed polynomial degree by an appropriate subpath decomposition; see
    Remark \ref{rmk:decomposition}.
    \end{remark}

    To ensure entropy stability, we suggest the following setup for the wave speed:
    $$\mathcal S_R=\max\{\mathcal S_R^{\mathrm{std}},\mathcal S_R^{\mathrm{ES}}\},\quad \mathcal S_L=\min\{\mathcal S_L^{\mathrm{std}},\mathcal S_L^{\mathrm{ES}}\},
    $$
    where $\mathcal S_R^{\mathrm{ES}},\mathcal S_{L}^{\mathrm{ES}}$ denote the right-hand sides in \eqref{eq:speed}.  
    For ease of notation, we denote the three-point scheme \eqref{eq:first2} as $\bar{\mathbf U}_i^{n+1}=\mathcal H(\bar{\mathbf U}_{i-1}^n,\bar{\mathbf U}_i^n,\bar{\mathbf U}_{i+1}^n;\Phi_{i-1,i},\Phi_{i,i+1};\lambda)$, 
    and its ES CFL condition \eqref{eq:CFL1st} as
    $ \lambda\le \mathscr T(\bar{\mathbf U}_{i-1}^n,\bar{\mathbf U}_i^n,\bar{\mathbf U}_{i+1}^n;\Phi_{i-1,i},\Phi_{i,i+1})$. 
    
	\section{High-order ES scheme in one dimension} \label{sec4}
    
    Now, we construct the high-order ES DG scheme for \eqref{eq:MHD1D}. Denote the finite element space 
    \begin{equation*}
    V_h^k=\{w(x): w(x)|_{I_i}\in \mathbb P^k(I_i),\ \forall i\}.
    \end{equation*}
    Then the vector-valued DG space $\mathbf V_h^k$ can be constructed by $\mathbf V_h^k=[V_h^k]^p$. The semi-discrete DG scheme reads: find $\mathbf U_h\in \mathbf V_{h}^k$, such that for any $\mathbf W\in \mathbf V_{h}^k$ and $I_i$,
    \begin{equation}\label{eq:DGsemi}
    \begin{aligned}
    \int_{I_i}\frac{\partial \mathbf U_h}{\partial t}\cdot\mathbf W\mathrm dx
    =&\,\int_{I_i}\mathbf F(\mathbf U_h)\cdot\frac{\partial \mathbf W}{\partial x}\mathrm dx-\hat{\mathbf F}_{i+\frac{1}{2}}\cdot\mathbf W_{i+\frac{1}{2}}^-+\hat{\mathbf F}_{i-\frac{1}{2}}\cdot\mathbf W_{i-\frac{1}{2}}^+
    \\  &-\mathcal D_{i+\frac{1}{2}}^-\cdot\mathbf W_{i+\frac{1}{2}}^--\mathcal D_{i-\frac{1}{2}}^+\cdot\mathbf W_{i-\frac{1}{2}}^+-\int_{I_i}\mathbf S(\mathbf U_h)\cdot\mathbf W\frac{\partial B_{x,h}}{\partial x}\mathrm dx.
    \end{aligned}\end{equation}
    Here, $\hat{\mathbf F}_{i+1/2}=\hat{\mathbf F}(\mathbf U_{i+1/2}^-,\mathbf U_{i+1/2}^+)$ and
    $$
    \mathcal D^\pm_{i+\frac{1}{2}}=\mathcal R^\pm_{i+\frac{1}{2}}\int_0^1\mathbf S(\Psi_{i+\frac{1}{2}}(\xi))B_x'\Psi_{i+\frac{1}{2}}'(\xi)\mathrm d\xi,\quad \mathcal R_{i+\frac{1}{2}}^\pm=\mathcal R^\pm(\mathbf U_{i+\frac{1}{2}}^-,\mathbf U_{i+\frac{1}{2}}^+),
    $$
    where $\Psi_{i+1/2}(\xi)=\Psi(\xi;\mathbf U_{i+1/2}^-,\mathbf U_{i+1/2}^+)$ is the path connecting $\mathbf U_{i+1/2}^-$ and $\mathbf U_{i+1/2}^+$. By using the method-of-lines framework, \eqref{eq:DGsemi} can be treated as an ODE system
    \begin{equation*}
    \frac{\mathrm d\mathbf U_h}{\mathrm dt}=\mathcal L_h(\mathbf U_h), 
    \end{equation*}
    where $\mathcal L_h: \mathbf V_h^k \to \mathbf V_h^k$ is defined by the right-hand side of \eqref{eq:DGsemi}. One can employ a time discretization method to obtain the fully-discrete scheme.

    Our subsequent entropy analysis will rely on the notion of Gauss--Lobatto quadrature, for which we denote by $\hat x_1,\dots,\hat x_N$ the quadrature points on $I_i$ and  $\omega_1+\dots+\omega_N=1$ the corresponding quadrature weights. We denote the $N$-point Gauss--Lobatto quadrature of $w(x)$ with $N\ge k+1$ on $I_i$ by $\int_{I_i}^{\langle N\rangle }w(x)\mathrm dx$ and define $\tilde{\mathcal U}_i^n=\frac{1}{h}\int_{I_i}^{\langle N\rangle}\mathcal U(\mathbf U_h^n)\mathrm dx$.
    
    Following our previous work \cite{liu2026limiter}, we will first prove that the forward Euler DG method satisfies a weak cell entropy inequality    \begin{equation}\label{eq:ESlike}\mathcal U(\bar{\mathbf U}_i^{n+1})\le \tilde{\mathcal U}_i^n-\lambda(\hat{\mathcal F}_{i+\frac{1}{2}}-\hat{\mathcal F}_{i-\frac{1}{2}}),
    \end{equation}
    then use the scaling limiter to enforce the genuine cell entropy inequality \begin{equation}\label{eq:ESdiscrete}\tilde{\mathcal U}_i^{n+1}\le \tilde{\mathcal U}_i^n-\lambda(\hat{\mathcal F}_{i+\frac{1}{2}}-\hat{\mathcal F}_{i-\frac{1}{2}}),\end{equation}
    and finally extend it to high-order time stepping with SSP multistep methods \cite{gottlieb2001strong}. 
    
    \subsection{A weak cell entropy inequality for cell averages} 
    The forward Euler fully-discrete scheme reads:
    \begin{equation}\label{eq:EFDG}
    \mathbf U_h^{n+1,(\mathrm{pre})}=\mathbf U_h^n+\Delta t\cdot\mathcal L_h(\mathbf U_h^n),
    \end{equation}
    Here, $\mathbf U_h^{n+1,(\mathrm{pre})}$ is a predicted state that will be limited to a new state $\mathbf U^{n+1}_h$ satisfying $\bar{\mathbf U}_i^{n+1} = \bar{\mathbf U}_i^{n+1,(\mathrm{pre})}$.  
        Therefore, 
    \begin{equation}\label{eq:EF1st}
    \begin{aligned}
    \bar{\mathbf U}_i^{n+1}=&\,\bar{\mathbf U}_i^n-\lambda(\hat{\mathbf F}_{i+\frac{1}{2}}-\hat{\mathbf F}_{i-\frac{1}{2}})
    -\lambda(\mathcal D_{i+\frac{1}{2}}^-+\mathcal D_{i-\frac{1}{2}}^+)
    -\lambda\int_{I_i}\mathbf S(\mathbf U_h^n)\frac{\partial B_{x,h}^n}{\partial x}\mathrm dx.
    \end{aligned}
    \end{equation}

Although the predicted stage $\mathbf U_h^{n+1,(\mathrm{pre})}$ may violate the genuine cell entropy inequality \eqref{eq:ESdiscrete}, it admits the weak cell entropy inequality \eqref{eq:ESlike} under the CFL condition in Theorem \ref{thm:ESlike1D}, which is one of the key contributions of our work. For homogeneous conservation laws, such a weak cell entropy inequality was first proved in \cite{carlier2023invariant}. However, for the MHD equations, the proof is nontrivial and relies on a careful prescription of generalized paths within the convex decomposition. 

\begin{theorem}\label{thm:ESlike1D}
The scheme \eqref{eq:EF1st} satisfies the weak cell entropy inequality
\eqref{eq:ESlike} in each of the following cases.

{(1) Distinct endpoint values:}
If $\mathbf U_{i-1/2}^+\ne\mathbf U_{i+1/2}^-$, then
\eqref{eq:ESlike} holds under
\begin{equation}\label{eq:CFL1D}
\begin{aligned}
\lambda&\le \omega_1
\mathscr T(\mathbf U_{i-\frac12}^-,
           \mathbf U_{i-\frac12}^+,
           \mathbf U_{i+\frac12}^-;
           \Psi_{i-\frac12},\Phi_i),\\
\lambda&\le \omega_N
\mathscr T(\mathbf U_{i-\frac12}^+,
           \mathbf U_{i+\frac12}^-,
           \mathbf U_{i+\frac12}^+;
           \Phi_i,\Psi_{i+\frac12}),
\end{aligned}
\end{equation}
where
\[
\Phi_i(\xi)=
\mathbf U_h^n|_{I_i}
\left(x_{i-\frac12}
+\xi(x_{i+\frac12}-x_{i-\frac12})\right).
\]

{(2) Coincident endpoint values:}
Suppose
$
\mathbf U_{i-1/2}^+
=\mathbf U_{i+1/2}^-=: \mathbf U_f$.

\begin{itemize}
\item[(2a)]
Let $\mathbf U_q:=\mathbf U_h^n(\hat x_q)$. If there exists $s\in\{2,\cdots,N-1\}$ such that
$\mathbf U_s\ne\mathbf U_f$, then \eqref{eq:ESlike} holds under
\begin{equation}\label{eq:CFL1D2}
\begin{aligned}
\lambda&\le \omega_1
\mathscr T(\mathbf U_{i-\frac12}^-,
           \mathbf U_{i-\frac12}^+,
           \mathbf U_s;
           \Psi_{i-\frac12},\Phi_i^L),\\
\lambda&\le \omega_s
\mathscr T(\mathbf U_{i-\frac12}^+,
           \mathbf U_s,
           \mathbf U_{i+\frac12}^-;
           \Phi_i^L,\Phi_i^R),\\
\lambda&\le \omega_N
\mathscr T(\mathbf U_s,
           \mathbf U_{i+\frac12}^-,
           \mathbf U_{i+\frac12}^+;
           \Phi_i^R,\Psi_{i+\frac12}),
\end{aligned}
\end{equation}
where $\Phi_i^L$ and $\Phi_i^R$ are the reparametrized portions of
$\Phi_i$ from $\mathbf U_{i-1/2}^+$ to $\mathbf U_s$ and from
$\mathbf U_s$ to $\mathbf U_{i+1/2}^-$, respectively.

\item[(2b)]
If no such $s$ exists, then
$\mathbf U_h^n|_{I_i}\equiv\mathbf U_f$, and \eqref{eq:ESlike} holds under
\begin{equation}\label{eq:CFL1Dconst}
\lambda\le
\mathscr T(\mathbf U_{i-\frac12}^-,
           \mathbf U_f,
           \mathbf U_{i+\frac12}^+;
           \Psi_{i-\frac12},\Psi_{i+\frac12}).
\end{equation}
If one of the two exterior states coincides with $\mathbf U_f$, the
corresponding zero-jump condition in Theorem \ref{thm:ES1st} is omitted;
if both coincide with $\mathbf U_f$, no CFL restriction is required.
\end{itemize}
\end{theorem}

\begin{proof}
 Since $N\ge k+1$, we have 
$
\bar{\mathbf U}_i^n=\sum_{q=1}^N\omega_q\mathbf U_q.$

\textnormal{(1)}
Set
\[
\hat{\mathbf F}_i
=\hat{\mathbf F}(\mathbf U_{i-\frac12}^+,\mathbf U_{i+\frac12}^-),
\qquad
\mathcal A_i^\pm
=\mathcal R_i^\pm
\int_0^1
\mathbf S(\Phi_i(\xi ))B_x'\Phi_i'(\xi)\,\mathrm d\xi ,
\]
where
$\mathcal R_i^\pm
=\mathcal R^\pm(\mathbf U_{i-1/2}^+,\mathbf U_{i+1/2}^-)$.
Since $\mathcal R_i^++\mathcal R_i^-=1$, the volume source term in \eqref{eq:EF1st} equals
$\mathcal A_i^++\mathcal A_i^-$. Hence, using $\bar{\mathbf U}_i^n=\sum_{q=1}^N\omega_q\mathbf U_q$,
\eqref{eq:EF1st} can be written as 
\begin{align*}
\bar{\mathbf U}_i^{n+1}
={}& \sum_{q=1}^{N}\omega_q\mathbf U_q -\lambda(\hat{\mathbf F}_{i+\frac{1}{2}}-\hat{\mathbf F}_{i-\frac{1}{2}})
    -\lambda(\mathcal D_{i+\frac{1}{2}}^-+\mathcal D_{i-\frac{1}{2}}^+)-\lambda(\mathcal A_{i}^-+\mathcal A_{i}^+)\\
=& \omega_1\left[
\mathbf{U}_{i-\frac12}^+ 
-\frac{\lambda}{\omega_1}
\bigl(\hat{\mathbf F}_i-\hat{\mathbf F}_{i-\frac12}\bigr)
-\frac{\lambda}{\omega_1}
\bigl(\mathcal A_i^-+\mathcal D^+_{i-\frac12}\bigr)
\right]
\\
&+
\omega_N\left[
\mathbf U_{i+\frac12}^-
-\frac{\lambda}{\omega_N}
\bigl(\hat{\mathbf F}_{i+\frac12}-\hat{\mathbf F}_i\bigr)
-\frac{\lambda}{\omega_N}
\bigl(\mathcal D^-_{i+\frac12}+\mathcal A_i^+\bigr)
\right]
+\sum_{q=2}^{N-1}\omega_q\mathbf U_q
\\
={}&
\omega_1\mathcal H_1+\omega_N\mathcal H_N
+\sum_{q=2}^{N-1}\omega_q\mathbf U_q ,
\end{align*}
with
\begin{align*}
\mathcal H_1=&
\mathcal H\left(
\mathbf U_{i-\frac12}^-,
\mathbf U_{i-\frac12}^+,
\mathbf U_{i+\frac12}^-;
\Psi_{i-\frac12},\Phi_i;
\frac{\lambda}{\omega_1}\right),\\
\mathcal H_N=&
\mathcal H\left(
\mathbf U_{i-\frac12}^+,
\mathbf U_{i+\frac12}^-,
\mathbf U_{i+\frac12}^+;
\Phi_i,\Psi_{i+\frac12};
\frac{\lambda}{\omega_N}\right).
\end{align*}
By the convexity of $\mathcal U$,
\[
\begin{aligned}
\mathcal U(\bar{\mathbf U}_i^{n+1})
&\le
\omega_1\mathcal U(\mathcal H_1)
+\omega_N\mathcal U(\mathcal H_N)
+\sum_{q=2}^{N-1}\omega_q\mathcal U(\mathbf U_q).
\end{aligned}
\]
Under the CFL conditions \eqref{eq:CFL1D}, Theorem~\ref{thm:ES1st}
applied to $\mathcal H_1$ and $\mathcal H_N$ gives
\[
\mathcal U(\mathcal H_1)
\le
\mathcal U(\mathbf U_1)
-\frac{\lambda}{\omega_1}
\left(
\hat{\mathcal F}_{i}
-\hat{\mathcal F}_{i-\frac12}
\right), \quad 
\mathcal U(\mathcal H_N)
\le
\mathcal U(\mathbf U_N)
-\frac{\lambda}{\omega_N}
\left(
\hat{\mathcal F}_{i+\frac12}
-\hat{\mathcal F}_{i}
\right),
\]
where
$\hat{\mathcal F}_{i}
=
\hat{\mathcal F}
(
\mathbf U_{i-1/2}^+,\mathbf U_{i+1/2}^-
)$ 
is the numerical entropy flux associated with the artificial interior interface.
Consequently,
\[
\begin{aligned}
\mathcal U(\bar{\mathbf U}_i^{n+1})
\le
\sum_{q=1}^{N}\omega_q\mathcal U(\mathbf U_q)
-\lambda
\left(
\hat{\mathcal F}_{i}
-\hat{\mathcal F}_{i-\frac12}
+
\hat{\mathcal F}_{i+\frac12}
-\hat{\mathcal F}_{i}
\right) =
\widetilde{\mathcal U}_i^n
-\lambda
\left(
\hat{\mathcal F}_{i+\frac12}
-\hat{\mathcal F}_{i-\frac12}
\right),
\end{aligned}
\]
which is precisely \eqref{eq:ESlike}.

\textnormal{(2a)}
For $D=L,R$, define
\[
\mathcal A_D^\pm
=
\mathcal R_D^\pm
\int_0^1
\mathbf S(\Phi_i^D(\xi))B_x'(\Phi_i^D(\xi))'\,\mathrm d\xi  \]
with 
\[
\mathcal R_L^\pm
=\mathcal R^\pm(\mathbf U_{i-\frac12}^+,\mathbf U_s),\quad   
\mathcal R_R^\pm
=\mathcal R^\pm(\mathbf U_s,\mathbf U_{i+\frac12}^-).\]
As $\Phi_i^L$ and $\Phi_i^R$ partition $\Phi_i$, the volume source
term is 
$\mathcal A_L^++\mathcal A_L^-
+\mathcal A_R^++\mathcal A_R^-$.  Introducing $
\hat{\mathbf F}_L
=\hat{\mathbf F}(\mathbf U_{i-1/2}^+,\mathbf U_s), \ 
\hat{\mathbf F}_R
=\hat{\mathbf F}(\mathbf U_s,\mathbf U_{i+1/2}^-)$, the update can be decomposed as
\[
\bar{\mathbf U}_i^{n+1}
=
\omega_1\mathcal H_1+\omega_s\mathcal H_s+\omega_N\mathcal H_N
+\sum_{q\ne1,s,N}\omega_q\mathbf U_q ,
\]
where $\mathcal H_1,\mathcal H_s,\mathcal H_N$ are precisely the three-point schemes
appearing in \eqref{eq:CFL1D2}. Applying
Theorem \ref{thm:ES1st} to these three terms and using convexity of
$\mathcal U$, the intermediate numerical entropy fluxes telescope, yielding
\eqref{eq:ESlike}.

\textnormal{(2b)}
If no such $s$ exists, then
$\mathbf U_q=\mathbf U_f$ for all $q=1,\ldots,N$. Since
$\mathbf U_h^n|_{I_i}\in[\mathbb P^k(I_i)]^p$ and $N\ge k+1$, polynomial
unisolvence implies
$
\mathbf U_h^n|_{I_i}\equiv\mathbf U_f$.
Thus the volume source term vanishes,
$\tilde{\mathcal U}_i^n=\mathcal U(\mathbf U_f)$, and
\eqref{eq:EF1st} reduces to the first-order three-point scheme
\[
\bar{\mathbf U}_i^{n+1}
=
\mathcal H(\mathbf U_{i-\frac12}^-,
           \mathbf U_f,
           \mathbf U_{i+\frac12}^+;
           \Psi_{i-\frac12},\Psi_{i+\frac12};\lambda).
\]
Theorem \ref{thm:ES1st} therefore gives \eqref{eq:ESlike} under
\eqref{eq:CFL1Dconst}; if one of the two jumps vanishes, the corresponding
term in the proof of Theorem \ref{thm:ES1st} is identically zero and its
wave-speed condition is vacuous.
\end{proof}

    \begin{remark}
    The proof of this theorem is one of the main highlights of this paper. \textit{The key to this proof is to notice that the internal source term integral can be naturally viewed as a generalized path integral connecting two internal states $\mathbf U_{i-1/2}^+,\mathbf U_{i+1/2}^-$ via the solution polynomial itself}, which facilitates the construction of the corresponding three-point scheme. We note again that the scheme \eqref{eq:DGsemi} remains the standard path-conservative DG scheme; the generalized form \eqref{eq:first2} is introduced only as an analytical tool for constructing the entropy inequalities of the three-point building blocks.
    \end{remark}
    \begin{remark}\label{rmk:decomposition}The decomposition in part {\rm (2a)} is not restricted to exactly coincident endpoint
values. For every fixed polynomial degree $k$, an appropriate decomposition of the
polynomial generalized path can be selected such that the variation of each subpath
is comparable to the distance between its endpoint states. For a subpath with endpoint difference $\mathbf d$, the usual boundedness assumptions
on the numerical states imply that the numerator in \eqref{eq:speed} is
$\mathcal O(\|\mathbf d\|^2)$, whereas the denominator is bounded below by
$C\|\mathbf d\|^2$ uniformly. Therefore, the corresponding entropy-stability
wave-speed bounds remain $\mathcal O(1)$, independently of the mesh size and of how
close the original endpoint states are. The decomposition is used only in the entropy analysis and does
not modify the implemented DG scheme.
    \end{remark} 

    \begin{remark}\label{rmk:path}
    It can be seen that the above entropy analysis is valid for any prescribed interface path $\Psi$ satisfying Definition \ref{def:pathfun}, as long as $\Phi_i$ is chosen as the polynomial generalized path induced by $\mathbf U_h^n$ in the proof.
    In other words, for any given path $\Psi$ that defines the non-conservative product in Godunov’s form of the MHD system, the scheme \eqref{eq:EF1st} formally preserves the path and weak entropy inequality. Theorem \ref{thm:ESlike1D} ensures that we can choose $\Psi$ as the physically relevant path without affecting the weak entropy stability. 
        \end{remark}
    \begin{remark}[Connections with other existing solvers]\label{rmk:other} For linear paths \eqref{eq:linpath}, note that $B_x'$ is constant, and thus
    $$B_x'(\mathbf U_{i+\frac{1}{2}}^+ -\mathbf  U_{i+\frac{1}{2}}^-) = B_{x,i+\frac{1}{2}}^+-B_{x,i+\frac{1}{2}}^-=:[B_{x,h}]_{i+\frac{1}{2}}.$$
    Then, scheme \eqref{eq:DGsemi} can be written as 
    \begin{align}\label{eq:DGseminew}
    \int_{I_i}\frac{\partial \mathbf U_h}{\partial t}\cdot\mathbf W\mathrm dx
     =&\,\int_{I_i}\mathbf F(\mathbf U_h)\cdot\frac{\partial \mathbf W}{\partial x}\mathrm dx-\hat{\mathbf F}_{i+\frac{1}{2}}\cdot\mathbf W_{i+\frac{1}{2}}^-+\hat{\mathbf F}_{i-\frac{1}{2}}\cdot\mathbf W_{i-\frac{1}{2}}^+
    \\  &-\mathcal R_{i+\frac{1}{2}}^-\bar{\mathbf S}_{i+\frac{1}{2}}\cdot\mathbf W_{i+\frac{1}{2}}^-[B_{x,h}]_{i+\frac{1}{2}}-\mathcal R_{i-\frac{1}{2}}^+\bar{\mathbf S}_{i-\frac{1}{2}}\cdot\mathbf W_{i-\frac{1}{2}}^+[B_{x,h}]_{i-\frac{1}{2}} \notag
    \\&-\int_{I_i}\mathbf S(\mathbf U_h)\cdot\mathbf W\frac{\partial B_{x,h}}{\partial x}\mathrm dx. \notag
    \end{align}
    Here, $\bar{\mathbf S}_{i+1/2}=\int_0^1\mathbf S(\mathbf U_{i+1/2}^-+\xi(\mathbf U_{i+1/2}^+-\mathbf U_{i+1/2}^-))\mathrm d\xi$. 
    Now, it can be seen that scheme \eqref{eq:DGseminew} is very similar to the classical DG schemes solving \eqref{eq:MHDsym} in \cite{liu2018entropy, wu2018provably, liu2025structureMHD}, where the only difference is that $\bar{\mathbf S}_{i\pm 1/2}$ is replaced by the single-sided limits $\mathbf S(\mathbf U_{i\pm 1/2}^\mp)$ in these works. 
    It is worth noting that these classical forms do not strictly satisfy \eqref{eq:pathcons}. Instead, we have
    $$\mathcal D_{i+\frac{1}{2}}^++\mathcal D_{i+\frac{1}{2}}^-=(\mathcal R_{i+\frac{1}{2}}^+\mathbf S(\mathbf U_{i+\frac{1}{2}}^+)+\mathcal R_{i+\frac{1}{2}}^-\mathbf S(\mathbf U_{i+\frac{1}{2}}^-))\int_0^1B_x'\Psi_{i+\frac{1}{2}}'(\xi)\mathrm d\xi,$$
    which is essentially a low-order approximation to the path integral for the linear path. Nevertheless, by using the same analytical framework, an analogue of Theorem \ref{thm:ES1st} can also be established for such schemes. Specifically, the entropy-stability wave-speed bounds in \eqref{eq:speed} can be modified to provide sufficient dissipation such that $T_1\le 0$ and $T_2\ge 0$, which yields the corresponding weak cell entropy inequality. We omit the detailed proof for brevity.
    \end{remark}

    The remaining parts of this section follow closely the work in \cite{liu2026limiter}.
    \subsection{Scaling limiter and genuine cell entropy inequality}

  Despite the gap between \eqref{eq:ESlike} and \eqref{eq:ESdiscrete}, an important result in \cite{chen2017entropy} implies that scaling the numerical solution toward its cell average 
    \begin{equation}\label{eq:limiter}
    \mathbf U_h^{n+1}|_{I_i}=\bar{\mathbf U}_i^{n+1}+\theta_i(\mathbf U_h^{n+1,(\mathrm{pre})}|_{I_i}-\bar{\mathbf U}_i^{n+1})
    \end{equation}
    will not increase the cell entropy. Consequently, we can achieve the genuine cell entropy inequality by scaling the solution towards its cell average. In our previous work \cite{liu2026limiter}, we proved the following lemma for hyperbolic conservation laws, which holds similarly in the present context.
    \begin{lemma}
    For the solution $\mathbf U_h^{n+1,(\mathrm{pre})}$ computed by \eqref{eq:EFDG}, denote
    $$
    \mathcal U^{\mathrm{high}}=\tilde{\mathcal U}_i^{n+1,(\mathrm{pre})},\quad \mathcal U^{\mathrm{up}}=\tilde{\mathcal U}_i^n-\lambda(\hat{\mathcal F}_{i+\frac{1}{2}}-\hat{\mathcal F}_{i-\frac{1}{2}}),\quad \mathcal U^{\mathrm{1st}}=\mathcal U(\bar{\mathbf U}_i^{n+1}),
    $$
    and let
    \begin{equation}\label{eq:theta}
    \theta_i=\min\left\{\frac{\mathcal U^{\mathrm{up}}-\mathcal U^{\mathrm{1st}}}{\mathcal U^{\mathrm{high}}-\mathcal U^{\mathrm{1st}}}, 1\right\}, \quad \text{with }\theta_i = 1 \text{ if }\  \mathcal U^{\mathrm{high}}=\mathcal U^{\mathrm{1st}}.
    \end{equation}
    Then the limited solution \eqref{eq:limiter} satisfies the genuine discrete cell entropy inequality \eqref{eq:ESdiscrete}. As a corollary, the solution is ES in the sense of
    $\sum\limits_{i}\tilde{\mathcal U}_i^{n+1}\le \sum\limits_{i}\tilde{\mathcal U}_i^n$.
    \end{lemma}

    The limiting coefficient \eqref{eq:theta} is an explicit and robust choice. Moreover, it is also high-order in space. Using a similar analysis to that in \cite{liu2026limiter}, under certain assumptions, it can be proven that when approximating smooth solutions on the cell $I_i$, the limiter will preserve the desired accuracy. 
    
    \subsection{High-order time discretization} 
    
    The previous two subsections focus on  forward Euler time discretization, which is only first-order accurate in time. To obtain uniformly high-order accuracy, we can replace the forward Euler time discretization with an SSP multistep method \cite{gottlieb2001strong}. An $m$-step, $r$th-order SSP multistep method has the form
\begin{equation}\label{eq:MS}
\mathbf U_h^{n+1,(\mathrm{pre})}=\sum\limits_{l=1}^m\left(\alpha_l\mathbf U_h^{n+1-l}+\beta_l\Delta t\cdot\mathcal L_h(\mathbf U_h^{n+1-l})\right),
\end{equation}
where $\alpha_l\ge 0$ for $1\le l\le m$ and $\sum\limits_{l=1}^m\alpha_l=1$. If $\beta_l<0$, the operator $\mathcal L_h$ should be replaced by the downwind operator $\tilde{\mathcal L}_h$. Then, we use the limiter \eqref{eq:limiter} to get $\mathbf U_h^{n+1}$. It can be seen that the SSP multistep method is a convex combination of forward Euler steps. Hence, we can get the following results.

\begin{theorem}\label{thm:fullyES-multistep}
If each forward Euler step in \eqref{eq:MS} satisfies the  corresponding CFL condition in Theorem \ref{thm:ESlike1D}, then \eqref{eq:MS} satisfies the weak cell entropy inequality
$$
\mathcal U(\bar{\mathbf U}_i^{n+1})\le \sum\limits_{l=1}^m\left( \alpha_l\,\tilde{\mathcal U}_i^{n+1-l}-\beta_l\lambda(\hat{\mathcal F}_{i+\frac{1}{2}}^{n+1-l}-\hat{\mathcal F}_{i-\frac{1}{2}}^{n+1-l})  \right).
$$
\end{theorem}

\begin{theorem}\label{thm:genuine-ES} For SSP multistep time discretization \eqref{eq:MS}, define
\begin{equation}\label{eq:thetaMS}
\mathcal U^{\mathrm{up}}=\sum\limits_{l=1}^m\left( \alpha_l\,\tilde{\mathcal U}_i^{n+1-l}-\beta_l\lambda(\hat{\mathcal F}_{i+\frac{1}{2}}^{n+1-l}-\hat{\mathcal F}_{i-\frac{1}{2}}^{n+1-l})  \right).
\end{equation}
Then under an appropriate CFL condition, the scheme \eqref{eq:MS} with limiter \eqref{eq:limiter} using \eqref{eq:thetaMS} satisfies the genuine cell entropy inequality in the sense of
\begin{equation}\label{eq:ESMS}
    \tilde{\mathcal U}_i^{n+1}\le \sum\limits_{l=1}^m\left( \alpha_l\,\tilde{\mathcal U}_i^{n+1-l}-\beta_l\lambda(\hat{\mathcal F}_{i+\frac{1}{2}}^{n+1-l}-\hat{\mathcal F}_{i-\frac{1}{2}}^{n+1-l})  \right),
\end{equation}
and is globally ES in the sense of 
$\sum\limits_i\tilde{\mathcal U}_i^{n+1}\le \sum\limits_{l=1}^{m}\alpha_l\left(\sum\limits_i \tilde{\mathcal U}_i^{n+1-l}\right)\le \max\limits_{1\le l\le m}\sum\limits_i \tilde{\mathcal U}_i^{n+1-l}$. 
\end{theorem}
For smooth solutions, following an analysis similar to that in \cite{liu2026limiter}, we can show that the scaling limiter does not affect the spatial accuracy under appropriate assumptions.  Therefore, for $\mathbb P^k$ approximations, choosing a $(k+1)$th order multistep method, and using $N=k+1$ Gauss--Lobatto points is typically sufficient to preserve the optimal accuracy. For more details, see \cite{liu2026limiter}.

\subsection{Entropy stability of the solution limit}  

It is well-known that the Lax--Wendroff theorem implies the solution limit  of a conservative scheme, if it exists, is a weak solution of hyperbolic conservation laws. Despite the known convergence issue to weak solutions \cite{liu2018entropy, wu2018provably, sun2023numerical, liu2025globally}, we can still prove that its limit solution satisfies the entropy condition in the weak sense \eqref{eq:ESweak}. We first introduce several assumptions.

\begin{assumption}[Convergence]\label{assu:conv}
The numerical solutions take values in a fixed compact subset of the physical admissible set. Moreover, $\mathbf{U}_h^0 \to \mathbf{U}_0$ in $L_\mathrm{loc}^1(\mathbb{R})$ as $\Delta x\to 0$ and $\mathbf{U}_h \to \mathbf{U}^\star$ in $L_{\mathrm{loc}}^1(\mathbb{R}\times \mathbb{R}^+)$ as $\Delta x, \Delta t \to 0$.
\end{assumption}

\begin{assumption}[TVB-like property]\label{assu:tvd}
				The numerical solution satisfies
				\begin{equation*}
					\sup\limits_{n}\Delta x\sum\limits_{i} \max\limits_{x\in B_i}\left\|\mathbf U_h^n(x)-\mathbf U_h^n(x_i)\right\|\to 0\quad \mathrm{as}\ \ \Delta x,\Delta t\to 0, \quad B_i=I_{i-1}\cup I_i\cup I_{i+1}.
				\end{equation*}
			\end{assumption}

Then, under these assumptions, the following convergence theorem can be established. The proof is similar to that in \cite{liu2026limiter} and uses a Lax--Wendroff argument. We omit it here for brevity. We only point out here that the discrete cell entropy inequality \eqref{eq:ESdiscrete} or \eqref{eq:ESMS} is the key to completing the proof.

\begin{theorem}\label{thm:LW}
Suppose the numerical solution computed by \eqref{eq:EFDG} or \eqref{eq:MS}, with the ES limiter \eqref{eq:limiter}, satisfies Assumptions \ref{assu:conv} and \ref{assu:tvd}. For the multistep scheme \eqref{eq:MS}, assume additionally that the starting values are consistent. For any smooth test function $\phi\in C_0^\infty(\mathbb R\times\mathbb R^+)$ with $\phi \ge 0$, the limit solution $\mathbf U^\star$ satisfies the entropy inequality in the sense of distributions
					\begin{equation*}
						\int_{\mathbb R^+}\int_{\mathbb R} \left(\mathcal U(\mathbf U^\star)\phi_t+\mathcal F(\mathbf U^\star)\phi_x\right) \mathrm dx\mathrm dt\ge -\int_{\mathbb R}\mathcal U(\mathbf U_0(x))\phi(x,0)\,\mathrm dx.
					\end{equation*}
\end{theorem}

\section{Extension to multi-dimensions} \label{sec5}
The 1D framework can be extended to multiple dimensions direction by direction. Here we consider the 2D case as an example:
\begin{equation*}
\frac{\partial\mathbf U}{\partial t}+\frac{\partial \mathbf F_1(\mathbf U)}{\partial x}+\frac{\partial \mathbf F_2(\mathbf U)}{\partial y}=-\mathbf S(\mathbf U)(\nabla\cdot\mathbf B).
\end{equation*}

\subsection{Semi-discrete DG scheme}
Assume a uniform partition of the spatial domain $\mathbb R^2$ into meshes $\mathcal K = \{K_{ij}\}$, where $K_{ij}=[x_{i-1/2},x_{i+1/2}]\times[y_{j-1/2},y_{j+1/2}]$ with grid sizes $h_x$ and $h_y$. For the 2D case, the finite element space is defined as 
$$V_h^k=\{w(x,y):w(x,y)|_{K_{ij}}\in \mathbb P^k(K_{ij}),\,\forall K_{ij}\in\mathcal K\}.$$
Let $\mathbf V_h^k=[V_h^k]^p$. Then, the semi-discrete DG scheme reads: Find $\mathbf U_h\in\mathbf V_{h}^k$, such that for any $\mathbf W\in\mathbf V_{h}^k$ and $K_{ij}\in \mathcal K$,
\begin{align}
\label{eq:semiDG2D}
\int_{K_{ij}}&{\frac{\partial \mathbf{U}_h}{\partial t}\cdot \mathbf{W}\mathrm{d}x\mathrm{d}y}=\int_{K_{ij}}{\mathbf{F}_1\left( \mathbf{U}_h \right) \cdot \frac{\partial \mathbf{W}}{\partial x}+\mathbf{F}_2\left( \mathbf{U}_h \right) \cdot \frac{\partial \mathbf{W}}{\partial y}\mathrm{d}x\mathrm{d}y}
\\
&-\int_{y_{j-\frac{1}{2}}}^{y_{j+\frac{1}{2}}}{\left( \hat{\mathbf{F}}_{1,i+\frac{1}{2}}\left( y \right) \cdot \mathbf{W}_{i+\frac{1}{2}}^{-}\left( y \right) -\hat{\mathbf{F}}_{1,i-\frac{1}{2}}\left( y \right) \cdot \mathbf{W}_{i-\frac{1}{2}}^{+}\left( y \right) \right) \mathrm{d}y} \notag
\\
&-\int_{x_{i-\frac{1}{2}}}^{x_{i+\frac{1}{2}}}{\left( \hat{\mathbf{F}}_{2,j+\frac{1}{2}}\left( x \right) \cdot \mathbf{W}_{j+\frac{1}{2}}^{-}\left( x \right) -\hat{\mathbf{F}}_{2,j-\frac{1}{2}}\left( x \right) \cdot \mathbf{W}_{j-\frac{1}{2}}^{+}\left( x \right) \right) \mathrm{d}x} \notag
\\
&-\int_{y_{j-\frac{1}{2}}}^{y_{j+\frac{1}{2}}}\left(\mathcal D_{1,i+\frac{1}{2}}^-(y)\cdot\mathbf W_{i+\frac{1}{2}}^-(y)+\mathcal D_{1,i-\frac{1}{2}}^+(y)\cdot\mathbf W_{i-\frac{1}{2}}^+(y)\right)\mathrm dy \notag
\\&-\int_{x_{i-\frac{1}{2}}}^{x_{i+\frac{1}{2}}}{\left( \mathcal D^-_{2,j+\frac{1}{2}}\left( x \right) \cdot \mathbf{W}_{j+\frac{1}{2}}^{-}\left( x \right) +\mathcal D_{2,j-\frac{1}{2}}^+\left( x \right) \cdot \mathbf{W}_{j-\frac{1}{2}}^{+}\left( x \right) \right) \mathrm{d}x} \notag
\\&-\int_{K_{ij}}\mathbf S(\mathbf U_h)\cdot\mathbf W(\nabla\cdot\mathbf B_h)\mathrm dx\mathrm dy. \notag
\end{align}
Here, 
\begin{align*}
\hat{\mathbf F}_{1,i+\frac{1}{2}}(y)&=\hat{\mathbf F}_1(\mathbf U_h^n(x_{i+\frac{1}{2}}^-,y),\mathbf U_h^n(x_{i+\frac{1}{2}}^+,y)),
\\ \mathcal D^\pm_{1,i+\frac{1}{2}}(y)&=\mathcal R^{x,\pm}_{i+\frac{1}{2}}(y)\int_0^1 \mathbf S(\Psi_{i+\frac{1}{2}}^y(\xi))B_x'\frac{\partial\Psi_{i+\frac{1}{2}}^y}{\partial\xi}(\xi)\mathrm d\xi,\\ \mathcal R_{i+\frac{1}{2}}^{x,\pm}(y)&=\mathcal R^{x,\pm}(\mathbf U_h^n(x_{i+\frac{1}{2}}^-,y),\mathbf U_h^n(x_{i+\frac{1}{2}}^+,y)),
\\ 
\Psi_{i+\frac{1}{2}}^y(\xi)&=\Psi(\xi;\mathbf U_h^n(x_{i+\frac{1}{2}}^-,y),\mathbf U_h^n(x_{i+\frac{1}{2}}^+,y)).
\end{align*}
The operator $\mathcal R^{x,\pm}$ represents \eqref{eq:R} in the $x$-direction. The notations in the $y$-direction are defined similarly. For \eqref{eq:semiDG2D}, the forward Euler fully-discrete DG scheme is also given by \eqref{eq:EFDG}.

\subsection{Fully-discrete ES scheme}
The cell average is updated with
\begin{equation}\label{eq:EFave2D}
\begin{aligned}\bar{\mathbf U}_{ij}^{n+1}&=\,\bar{\mathbf U}_{ij}^n-\frac{\lambda_x}{h_y}\int_{y_{j-\frac{1}{2}}}^{y_{j+\frac{1}{2}}}(\hat{\mathbf F}_{1,i+\frac{1}{2}}-\hat{\mathbf F}_{1,i-\frac{1}{2}})\mathrm dy-\frac{\lambda_y}{h_x}\int_{x_{i-\frac{1}{2}}}^{x_{i+\frac{1}{2}}}(\hat{\mathbf F}_{2,j+\frac{1}{2}}-\hat{\mathbf F}_{2,j-\frac{1}{2}})\mathrm dx 
\\
&-\frac{\lambda_x}{h_y}\int_{y_{j-\frac{1}{2}}}^{y_{j+\frac{1}{2}}}(\mathcal D_{1,i+\frac{1}{2}}^-+\mathcal D_{1,i-\frac{1}{2}}^+)\mathrm dy
-\frac{\lambda_y}{h_x}\int_{x_{i-\frac{1}{2}}}^{x_{i+\frac{1}{2}}}(\mathcal D_{2,j+\frac{1}{2}}^-+\mathcal D_{2,j-\frac{1}{2}}^+)\mathrm dx 
\\&-\frac{\Delta t}{h_xh_y}\int_{K_{ij}}\mathbf S(\mathbf U_h^n)(\nabla\cdot\mathbf B_h^n)\mathrm dx\mathrm dy, 
\end{aligned}
\end{equation}
where $\lambda_x = \Delta t/h_x$ and $\lambda_y = \Delta t/h_y$. For the 2D case, define
\begin{equation}\label{eq:Uup2D}
\mathcal U^{\mathrm{up}}=\,\tilde{\mathcal U}_{ij}^{n}-\frac{\lambda_x}{h_y}\int_{y_{j-\frac{1}{2}}}^{y_{j+\frac{1}{2}}}(\hat{\mathcal F}_{1,i+\frac{1}{2}}-\hat{\mathcal F}_{1,i-\frac{1}{2}})\mathrm dy
-\frac{\lambda_y}{h_x}\int_{x_{i-\frac{1}{2}}}^{x_{i+\frac{1}{2}}}(\hat{\mathcal F}_{2,j+\frac{1}{2}}-\hat{\mathcal F}_{2,j-\frac{1}{2}})\mathrm dx
\end{equation}
with the discrete cell entropy average
$$ \tilde{\mathcal U}_{ij}^n=\sum\limits_{q=1}^N \left\{\frac{\omega_q}{2h_y}\int_{y_{j-\frac{1}{2}}}^{y_{j+\frac{1}{2}}}\mathcal U(\mathbf U_h^n(\hat x_q,y))\mathrm dy+\frac{\omega_q}{2h_x}\int_{x_{i-\frac{1}{2}}}^{x_{i+\frac{1}{2}}}\mathcal U(\mathbf U_h^n(x,\hat y_q))\right\}\mathrm dx. $$
The entropy-stability result is stated as follows. The proof uses a technique similar to that in the 1D case, and we omit it here for brevity.

\begin{theorem}\label{thm:ESlike2D}
The scheme \eqref{eq:EFave2D} satisfies the following weak entropy inequality
\begin{equation*}
\mathcal U(\bar{\mathbf U}_{ij}^{n+1})\le \mathcal U^{\mathrm{up}}
\end{equation*}
under the CFL condition
\begin{align*}
\lambda_x&\le \frac{\omega_1}{2}\min\limits_{y\in[y_{j-\frac{1}{2}},y_{j+\frac{1}{2}}]}\mathscr T_x(\mathbf U_h^n(x_{i-\frac{1}{2}}^-,y),\mathbf U_h^n(x_{i-\frac{1}{2}}^+,y),\mathbf U_h^n(x_{i+\frac{1}{2}}^-,y);\Psi_{i-\frac{1}{2}}^y,\Phi_i^y), 
\\ \lambda_x&\le \frac{\omega_N}{2}\min\limits_{y\in[y_{j-\frac{1}{2}},y_{j+\frac{1}{2}}]}\mathscr T_x(\mathbf U_h^n(x_{i-\frac{1}{2}}^+,y),\mathbf U_h^n(x_{i+\frac{1}{2}}^-,y),\mathbf U_h^n(x_{i+\frac{1}{2}}^+,y);\Phi_i^y,\Psi_{i+\frac{1}{2}}^y),    
\\ \lambda_y&\le \frac{\omega_1}{2}\min\limits_{x\in[x_{i-\frac{1}{2}},x_{i+\frac{1}{2}}]}\mathscr T_y(\mathbf U_h^n(x,y_{j-\frac{1}{2}}^-),\mathbf U_h^n(x,y_{j-\frac{1}{2}}^+),\mathbf U_h^n(x,y_{j+\frac{1}{2}}^-);\Psi_{j-\frac{1}{2}}^x,\Phi_{j}^x),
\\ \lambda_y&\le \frac{\omega_N}{2}\min\limits_{x\in[x_{i-\frac{1}{2}},x_{i+\frac{1}{2}}]}\mathscr T_y(\mathbf U_h^n(x,y_{j-\frac{1}{2}}^+),\mathbf U_h^n(x,y_{j+\frac{1}{2}}^-),\mathbf U_h^n(x,y_{j+\frac{1}{2}}^+);\Phi_j^x,\Psi_{j+\frac{1}{2}}^x).
\end{align*}
Here, $\mathscr T_x$ and $\mathscr T_y$ denote the CFL condition \eqref{eq:CFL1st} in the corresponding direction. Moreover, the generalized paths $\Phi_i^y$ and $\Phi_j^x$ connect
$\mathbf U_h^n(x_{i-1/2}^+,y)$ to $\mathbf U_h^n(x_{i+1/2}^-,y)$ and
$\mathbf U_h^n(x,y_{j-1/2}^+)$ to $\mathbf U_h^n(x,y_{j+1/2}^-)$, respectively,
and are constructed from $\mathbf U_h^n$ itself:
\begin{align*}\Phi_i^y(\xi)&=\mathbf U_h^n|_{K_{ij}}(x_{i-\frac{1}{2}}+\xi(x_{i+\frac{1}{2}}-x_{i-\frac{1}{2}}),y),
\\   
\Phi_j^x(\xi)&=\mathbf U_h^n|_{K_{ij}}(x,y_{j-\frac{1}{2}}+\xi(y_{j+\frac{1}{2}}-y_{j-\frac{1}{2}})).
\end{align*}
If directional endpoint states coincide, the corresponding CFL condition is understood using the decomposition in Theorem \ref{thm:ESlike1D}.
\end{theorem}
 
\begin{remark}
    For notational simplicity, Theorem \ref{thm:ESlike2D} is stated using the unsplit
directional generalized paths. For each fixed transverse coordinate,
however, $\Phi_i^y$ and $\Phi_j^x$ are one-dimensional polynomial
generalized paths of degree at most $k$. Therefore, if their endpoint
states are coincident or arbitrarily close, the decomposition described
in Section~4 can be applied independently to each directional path.
For every fixed $k$, the resulting entropy-stability wave-speed bounds
remain uniformly $\mathcal O(1)$ under the same boundedness assumptions
as in Remark \ref{rmk:decomposition}. We omit the expanded
multidimensional CFL conditions to avoid cumbersome notation.
\end{remark} 

The ES limiter can be applied in the same way as in \eqref{eq:limiter} and \eqref{eq:theta}. As a result, the fully-discrete ES property can be established as follows.

\begin{theorem}
The scheme \eqref{eq:semiDG2D} with forward Euler time discretization \eqref{eq:EFDG} and ES limiter \eqref{eq:limiter} using $\mathcal U^{\mathrm{up}}$ in \eqref{eq:Uup2D} 
satisfies the discrete cell entropy inequality
\begin{align*}
\tilde{\mathcal U}^{n+1}_{ij}\le &\,\,\tilde{\mathcal U}_{ij}^{n}-\frac{\lambda_x}{h_y}\int_{y_{j-\frac{1}{2}}}^{y_{j+\frac{1}{2}}}(\hat{\mathcal F}_{1,i+\frac{1}{2}}-\hat{\mathcal F}_{1,i-\frac{1}{2}})\mathrm dy-\frac{\lambda_y}{h_x}\int_{x_{i-\frac{1}{2}}}^{x_{i+\frac{1}{2}}}(\hat{\mathcal F}_{2,j+\frac{1}{2}}-\hat{\mathcal F}_{2,j-\frac{1}{2}})\mathrm dx,
\end{align*}
and is globally ES in the sense of 
$\sum\limits_{i,j}\tilde{\mathcal U}_{ij}^{n+1}\le \sum\limits_{i,j}\tilde{\mathcal U}_{ij}^n. $
\end{theorem}

The extension to high-order temporal accuracy with an SSP multistep method is also similar to the 1D case, and we thus omit it for brevity.

\subsection{Locally divergence-free methods}\label{sec:LDF}  
Although the scheme \eqref{eq:semiDG2D} with the limiter \eqref{eq:limiter} is entropy stable, we still need to control the divergence of the magnetic field. To address this issue, four main approaches have been proposed in the literature: 8-wave formulation \cite{powell1999solution}, divergence-cleaning methods \cite{rueda2023entropy, chen2025new}, projection methods \cite{toth2000b}, and constrained transport (CT) techniques \cite{balsara2004second, christlieb2014finite}. Within the framework of DG methods, the LDF approach, introduced by Cockburn, Li, and Shu \cite{cockburn2004locally, li2005locally}, serves as a natural candidate. The LDF method employs a specialized vector finite element space for numerical magnetic field $\mathbf B_h$, ensuring that it is strictly divergence-free within each cell.  Specifically, the finite element space is defined as 
\begin{equation}\label{eq:LDFspace}
\mathbf V_{h,0}^k=\left\{ \mathbf W\in\mathbf V_h^k:\nabla\cdot\mathbf B(\mathbf W)=0,\ \forall K_{ij}\in\mathcal K \right\}.
\end{equation}
A detailed construction of the LDF space with $d=2,3$ can be found in \cite{cockburn2004locally}. By utilizing the LDF space, the semi-discrete scheme reads: Find $\mathbf U_h\in \mathbf V_{h,0}^k$, such that for any $\mathbf W\in \mathbf V_{h,0}^k$ and $K_{ij}\in \mathcal K$, \eqref{eq:semiDG2D} holds. Note that since $\nabla\cdot\mathbf B_h=0$ everywhere in $K_{ij}$, the last interior integral term in \eqref{eq:semiDG2D} vanishes.  Fortunately, since a piecewise constant polynomial is naturally LDF, it follows that the cell average still satisfies \eqref{eq:EFave2D}. According to the above analysis, Theorem \ref{thm:ESlike2D} also holds for the forward Euler DG scheme with LDF space. Then, the limiter \eqref{eq:limiter} can be applied to preserve the genuine cell entropy inequality analogously. Moreover, the application of this limiter does not affect the divergence-free property of the magnetic field.

\section{Numerical tests}\label{sec6}

In this section, we present several numerical examples to demonstrate the performance of the proposed scheme. We use a $\mathbb P^2$ approximation and a 6-step, third-order SSP multistep method
$$
\mathbf U_h^{n+1}=\frac{108}{125}\mathbf U_h^n+ \frac{36}{25}\Delta t\cdot\mathcal L_h(\mathbf U_h^n)+\frac{17}{125}\mathbf U_h^{n-5}+\frac{6}{25}\Delta t\cdot\mathcal L_h(\mathbf U_h^{n-5})
$$
with $\mathrm{CFL}=0.06$ for all simulations. To demonstrate the advantage of the fully-discrete entropy stability, we do not apply any other slope limiter (e.g. the TVB limiter \cite{cockburn1989tvb3}, WENO limiter \cite{qiu2005runge}, COS limiter \cite{cao2026cos}) to the solution for all tests. However, for extreme examples, such as strong shocks or low density/pressure regions, a positivity-preserving (PP) limiter \cite{wu2018provably} is still needed to maintain the physical admissibility of the numerical solutions. As mentioned in \cite{chen2017entropy, liu2026limiter}, it will not increase the cell entropy.  To verify the entropy stability of the scheme, we will compute the maximum cell entropy violation of the solution for some examples, defined as $\max_{i,j}(\tilde{\mathcal U}_{ij}-\mathcal U^{\mathrm{up}}|_{K_{ij}})$. Obviously, this quantity should be non-positive for a fully-discrete ES scheme. Moreover, to measure the global divergence, we also focus on the divergence norm introduced in \cite{cockburn2004locally}, which is defined as
$$
\left\|\mathrm{div}\mathbf B_h\right\|=\sum\limits_{K_{ij}\in\mathcal K}\left(\int_{K_{ij}}\left|\nabla\cdot\mathbf B_h\right|\mathrm dx\mathrm dy+\int_{\partial K_{ij}}\left|[\mathbf B_h\cdot\mathbf n]\right|\mathrm ds\right).
$$
This quantity should be $k$th order for smooth solutions \cite{cockburn2004locally}. We will compare the performance of six variations, defined as: 
\begin{itemize}[leftmargin=*]
    \item \textbf{Base}: Standard DG scheme directly solving \eqref{eq:MHD} without Godunov's source term.
    \item \textbf{SG}: Scheme \eqref{eq:semiDG2D} without adding the ES limiter \eqref{eq:limiter}.
    \item \textbf{ES}: Scheme \eqref{eq:semiDG2D} with ES limiter \eqref{eq:limiter}.
    \item \textbf{LDF}: Standard LDF-DG scheme \cite{li2005locally} using the LDF space \eqref{eq:LDFspace} and directly solving \eqref{eq:MHD} without Godunov's source term.
    \item \textbf{SG-LDF}: Scheme \eqref{eq:semiDG2D} with LDF space \eqref{eq:LDFspace} but without adding ES limiter \eqref{eq:limiter}.
    \item \textbf{ES-LDF}: Proposed scheme, i.e. the scheme \eqref{eq:semiDG2D} with LDF space \eqref{eq:LDFspace} and ES limiter \eqref{eq:limiter}.
\end{itemize}

\begin{Ex}[Smooth MHD vortex]\label{ex:vortex}
We first consider the smooth vortex test problem, which was first introduced in \cite{balsara2004second}. This genuinely nonlinear benchmark is usually used to verify the accuracy of numerical schemes. We run this problem until $T=20$ on different mesh sizes with $N_x=N_y=:N$. Fig. \ref{figvortex1} illustrates the $L^2$ and $L^\infty$ errors for $B_x$. It is observed that all tested schemes achieve the optimal third-order accuracy for $k=2$. Furthermore, one can see that both Godunov's source term and the LDF treatment can effectively reduce the error of $B_x$ relative to the Base scheme.

We also plot the divergence norm and maximum cell entropy inequality violation in Fig. \ref{figvortex2}. From Fig. \ref{figvortex2} (a), the divergence norms of all schemes except the Base scheme converge at second order. Meanwhile, the schemes equipped with the LDF property (LDF, SG-LDF, ES-LDF) yield smaller divergence errors. From Fig. \ref{figvortex2}(b), we can see the schemes without Godunov's source term (Base, LDF) have the largest violations of the cell entropy inequality. The schemes solving \eqref{eq:MHDsym} with the source term (SG, SG-LDF) show better performance, and do not violate the cell entropy inequality for $N=25, 50$. However, for $N=100, 200$, these schemes are still not ES. Nevertheless, their entropy violations are smaller than the schemes without source term. Among all schemes, the schemes solving \eqref{eq:MHDsym} with ES limiter (ES, ES-LDF) keep the violation at the level of machine precision, confirming their strict fully-discrete entropy stability.
\end{Ex}

\begin{figure}[htb!]
	\centering
	\subfigure[$L^2$ error.]{
		\includegraphics[width=0.4\linewidth]{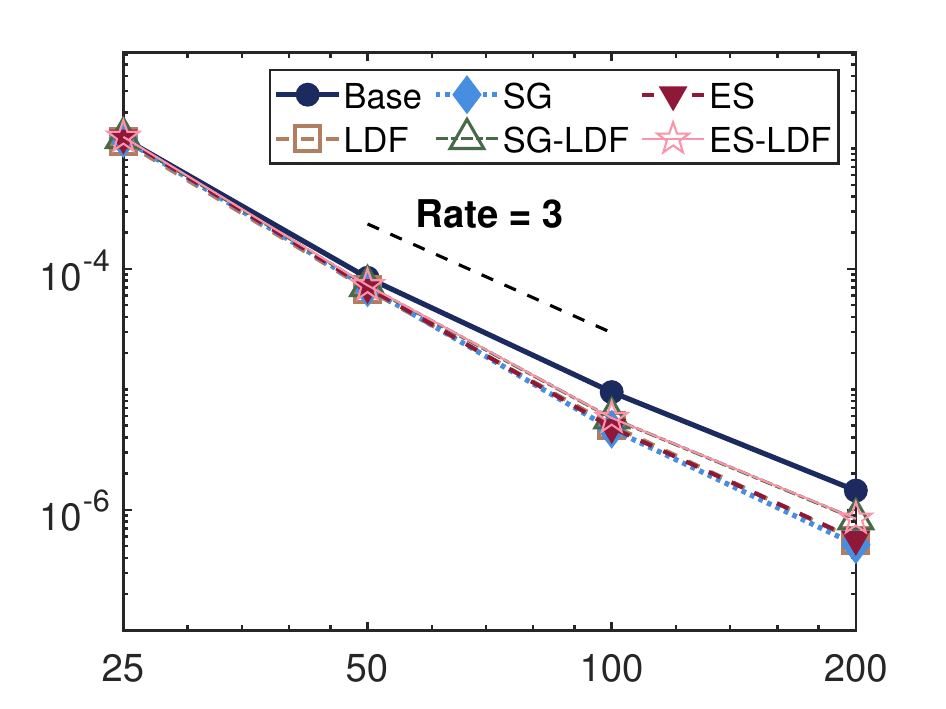}}
    \subfigure[$L^\infty$ error.]{
		\includegraphics[width=0.4\linewidth]{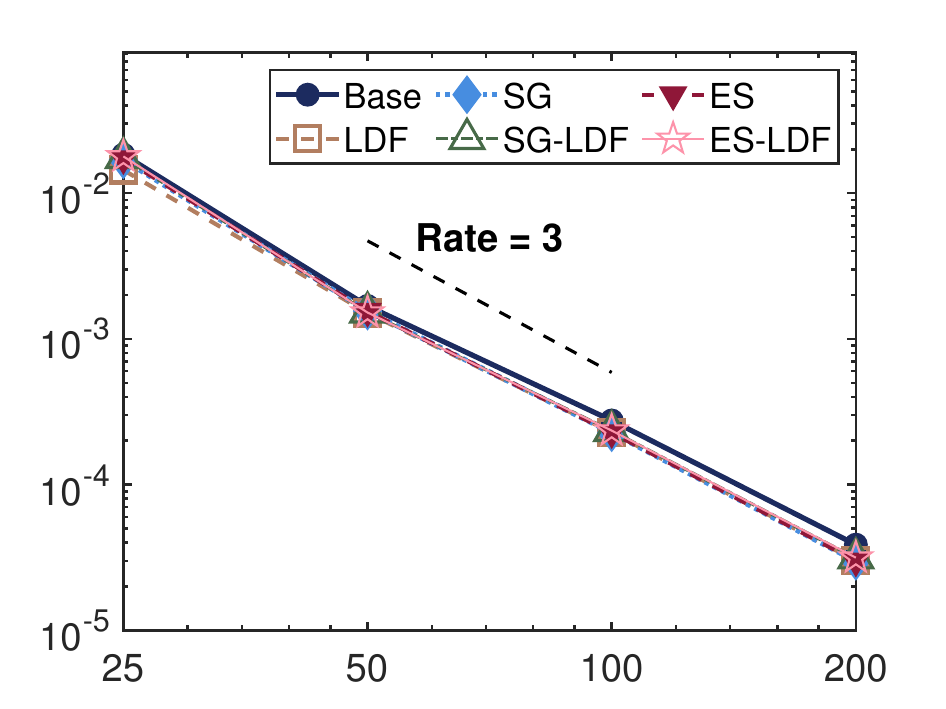}}
        \caption{Example \ref{ex:vortex}: Smooth MHD vortex. The $L^2$ and $L^\infty$ errors of $B_x$ and convergence rates of different schemes.}
        \label{figvortex1}
\end{figure}
        
\begin{figure}[htb!]
    \centering
	\subfigure[$\left\|\mathrm{div}\mathbf B_h\right\|$.]{
	   \includegraphics[width=0.4\linewidth]{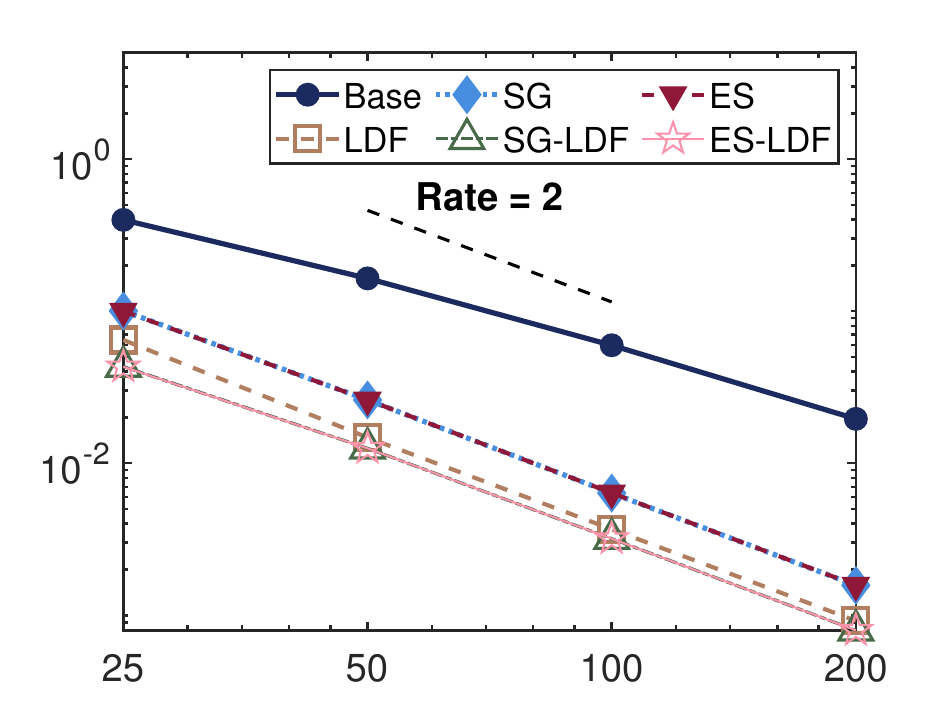}}
    \subfigure[$\max\limits_{i,j}(\tilde{\mathcal U}_{ij}-\mathcal U^{\mathrm{up}}|_{K_{ij}})$.]{
		\includegraphics[width=0.4\linewidth]{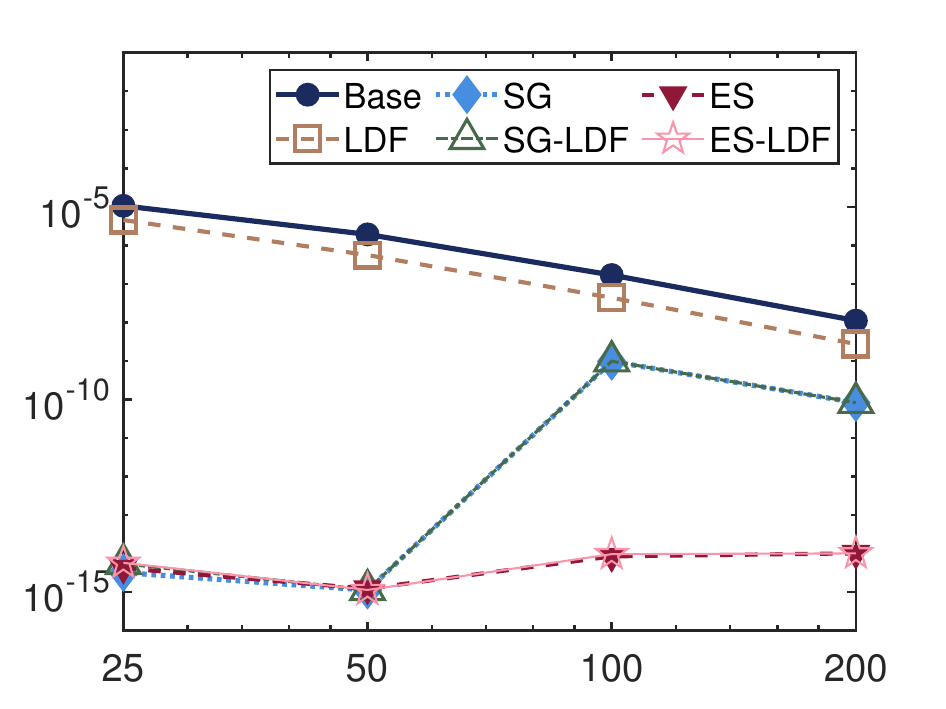}}
	\caption{Example \ref{ex:vortex}: Smooth MHD vortex. The divergence norm and maximum cell entropy inequality violation. }
    \label{figvortex2}
\end{figure}

\begin{Ex}[Orszag--Tang vortex]\label{ex:OTV} The Orszag--Tang vortex \cite{orszag1979small} is a classical test for evaluating the robustness of MHD schemes as the flow transitions from smooth initial conditions to complex states with interacting shock waves. We simulate all schemes to $T=0.5$ on $N_x\times N_y=192\times 192$ meshes, and the cell entropy violation is plotted in Fig. \ref{figOTV2}. The results of SG and SG-LDF schemes are similar, and likewise for ES and ES-LDF schemes. This verifies the advantage of the proposed scheme in strictly preserving the fully-discrete entropy stability property.

\begin{figure}[htb!]
	\centering
	\subfigure[Base.]{
		\includegraphics[width=0.4\linewidth]{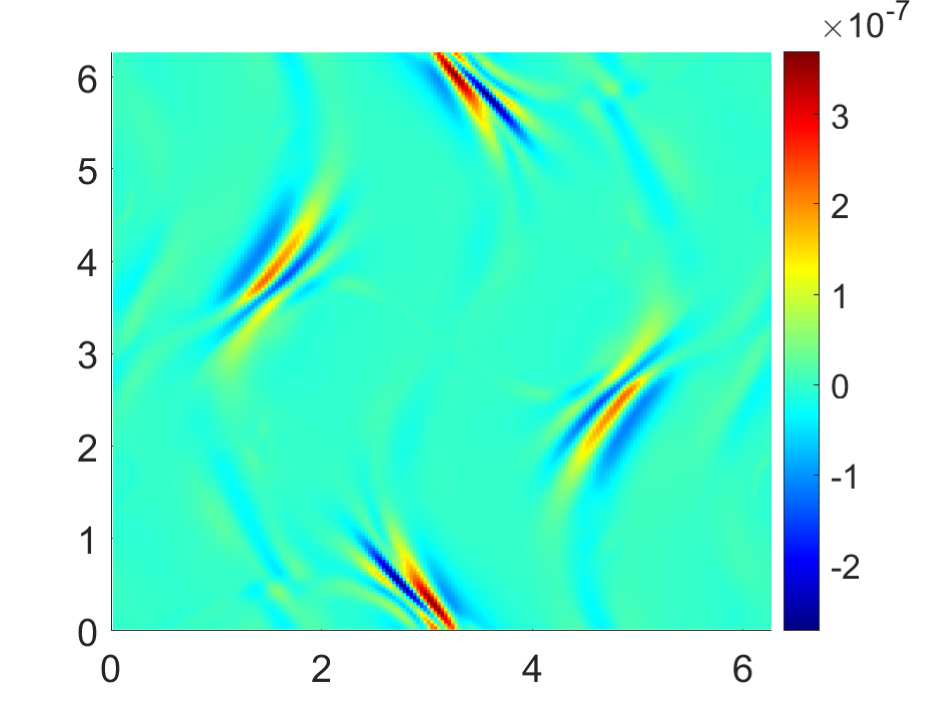}}
    \subfigure[LDF.]{
		\includegraphics[width=0.4\linewidth]{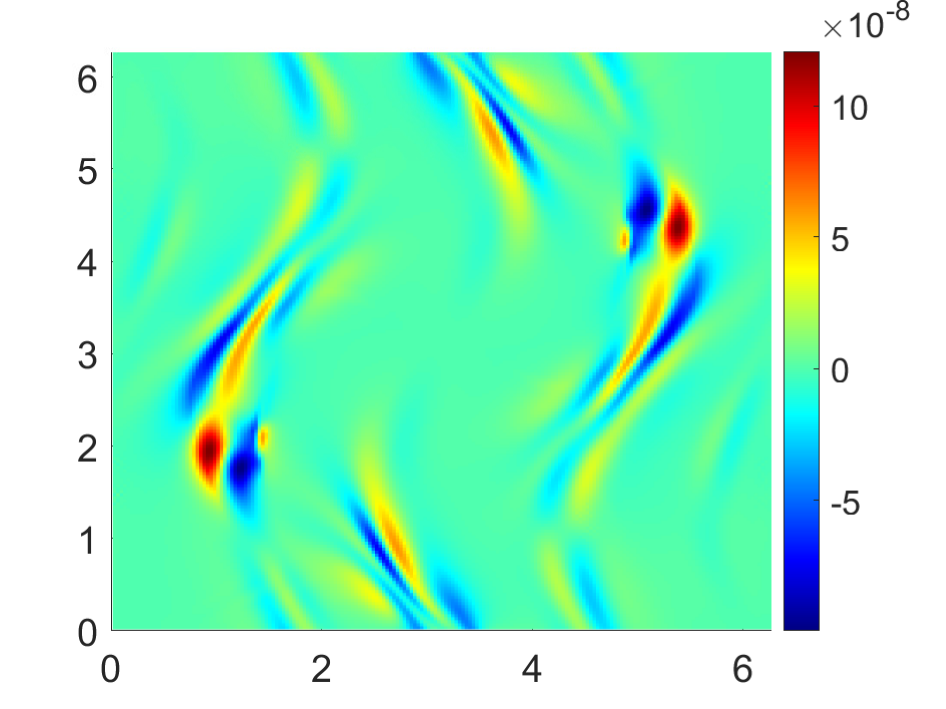}}
	\subfigure[SG-LDF.]{
	\includegraphics[width=0.4\linewidth]{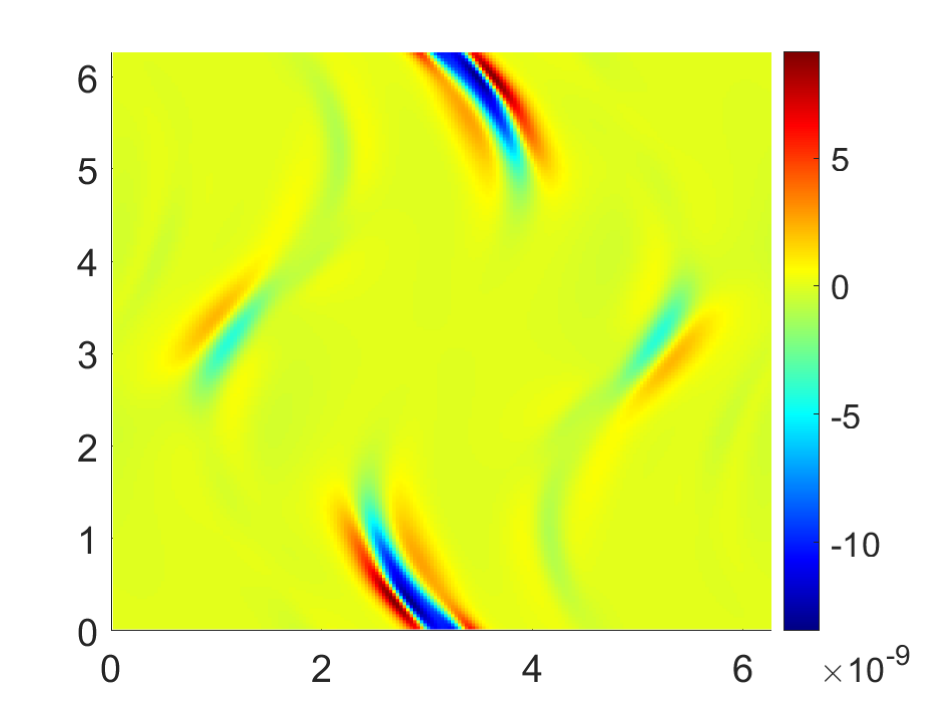}}
    \subfigure[ES-LDF.]{
		\includegraphics[width=0.4\linewidth]{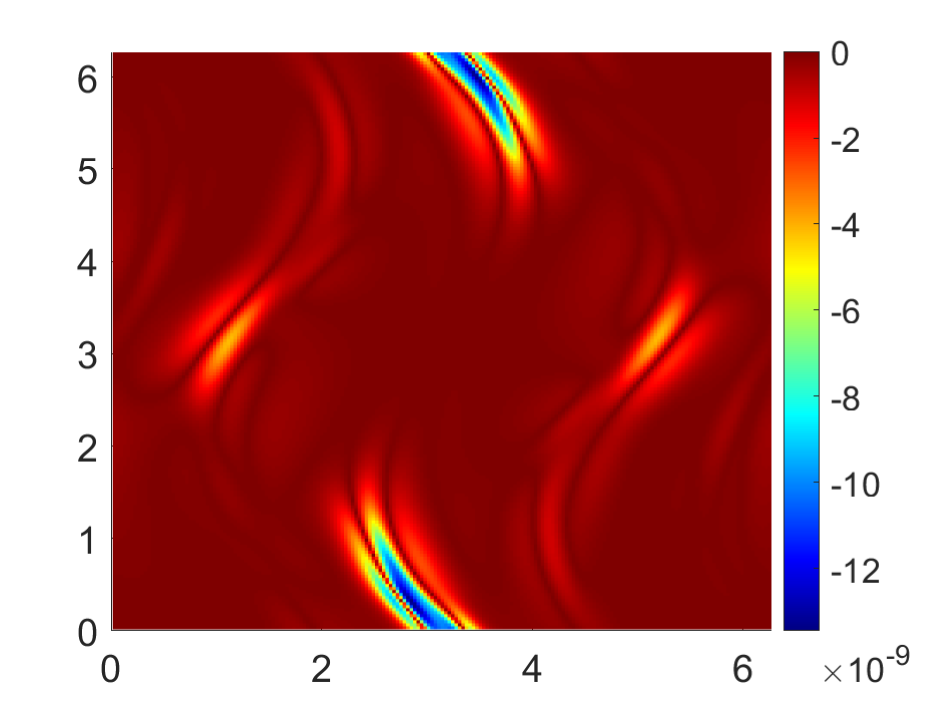}}
	\caption{Example \ref{ex:OTV}: Orszag--Tang vortex. The cell entropy violation at $T=0.5$. }
    \label{figOTV2}
\end{figure}

Next, we simulate this problem to longer times. When $T\approx 1.08$, the shock appears, and the non-ES schemes (Base, LDF, SG, SG-LDF) all blow up near this time, while the ES and ES-LDF schemes remain stable. The results of these two schemes are similar. In Fig. \ref{figOTV}, we only show the density results for the ES-LDF scheme at $T=3,4$. These are in good agreement with the results in \cite{li2005locally, liu2025globally, liu2025structureMHD}, and due to its low-dissipation nature, the small structures are well captured. For this example, the semi-discrete ES scheme \cite{liu2018entropy} will also blow up without a shock limiter, highlighting the advantage of maintaining fully-discrete entropy stability.

\begin{figure}[htb!]
	\centering
        
	\subfigure[$T=3$.]{
	\includegraphics[width=0.4\linewidth]{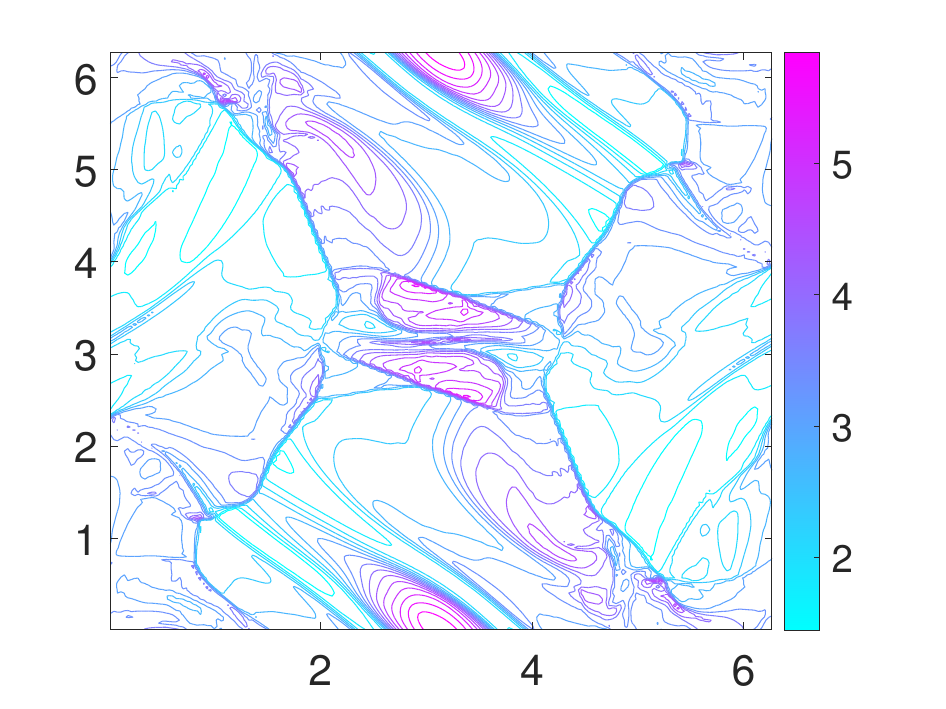}}
    \subfigure[$T=4$.]{
		\includegraphics[width=0.4\linewidth]{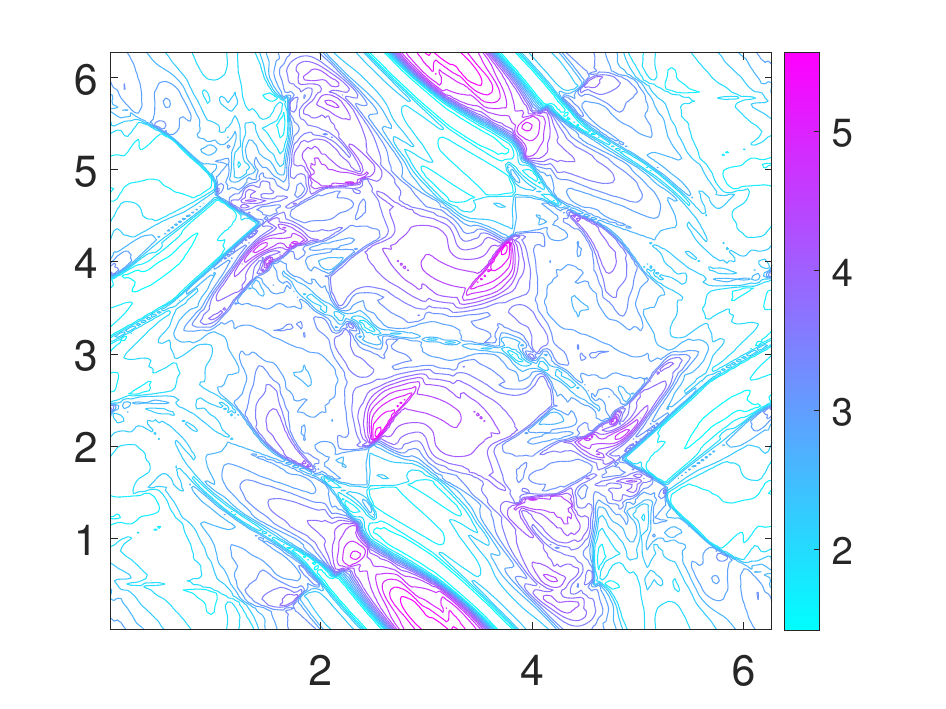}}
	\caption{Example \ref{ex:OTV}: Orszag--Tang vortex. $N_x\times N_y=192\times 192$ meshes. 
    }
    \label{figOTV}
\end{figure}

\end{Ex}

\begin{Ex}[Kelvin--Helmholtz instability]\label{ex:KH} The Kelvin--Helmholtz instability is a fundamental fluid instability. This test assesses the scheme's ability to resolve complex, small-scale turbulent structures driven by shear flows. Here we follow the setup in \cite{mignone2010high}. For $t \le 5$, the perturbation undergoes linear growth, winding up magnetic field lines as an anomalous cat's eye vortex forms. Conversely, for $t \ge 8$, tearing mode instabilities suppress field amplification, initiating magnetic reconnection that drives flux out of the vortex.

We use $N_x\times N_y=256\times 512$ meshes to simulate this problem. In Fig. \ref{figKH}, we display the results of $B_p/B_t$ at $T=20$, where $B_p=\sqrt{B_x^2+B_y^2}$ and $B_t=B_z$. The Base scheme will blow up near $T=4.8$, and the other non-ES schemes (LDF, SG, SG-LDF) will blow up near $T=7.8$. Thanks to the fully-discrete entropy stability, the ES and ES-LDF schemes remain stable until $T=20$. Meanwhile, the ES-LDF scheme exhibits significantly less numerical dissipation, successfully capturing the fine-scale turbulent structures and the sharp rollup of the cat's eye vortices compared to the heavily smeared results of the ES scheme. Furthermore, the temporal evolution of the poloidal magnetic energy 
$<B_p^2>={\int_{\Omega}B_p^2(t)\mathrm dx\mathrm dy}/{\int_{\Omega} B_p^2(0)\mathrm dx\mathrm dy}$ 
in Fig. \ref{figKH2} (a) quantitatively demonstrates that the ES-LDF scheme predicts a higher peak energy, accurately reflecting the physical magnetic field amplification before the onset of tearing mode instabilities around $t\approx 8$. Notably, as observed in Fig. 6(a), the ES scheme exhibits an unnatural flattening in the growth rate of $<B_p^2>$ around $t\approx 5$. This non-physical artifact, which is absent in standard reference solutions in the literature \cite{mignone2010high, rueda2023entropy, liu2025globally}, suggests that the accumulation of divergence errors in the ES scheme introduces spurious magnetic forces. By effectively enforcing the LDF property, the ES-LDF scheme successfully eliminates this numerical artifact, recovering the smooth and uninterrupted magnetic field amplification characteristic of the instability. Meanwhile, the nearly identical evolution of $\Delta u_y=(u_{y,\mathrm{max}}-u_{y,\mathrm{min}})/2$ during the early linear growth phase in Fig. \ref{figKH2} (b) confirms that both schemes maintain high physical fidelity before the highly nonlinear turbulent mixing dominates.

\begin{figure}[htb!]
	\centering
	\subfigure[ES.]{
	\includegraphics[width=0.4\linewidth]{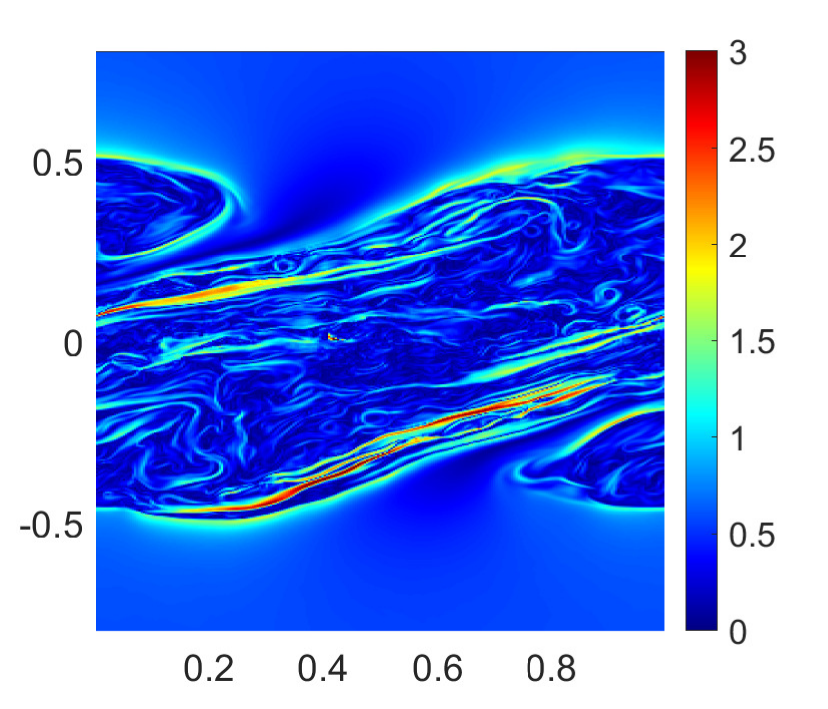}}
    \subfigure[ES-LDF.]{
	\includegraphics[width=0.4\linewidth]{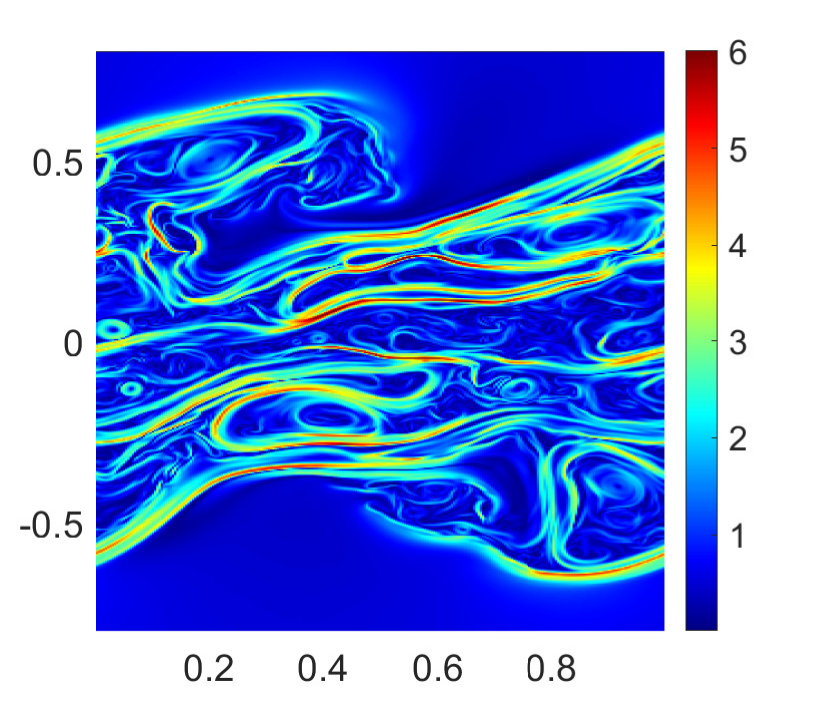}}
	\caption{Example \ref{ex:KH}: Kelvin--Helmholtz instability. $B_p/B_t$ at $T=20$. $N_x\times N_y=256\times 512$. }
    \label{figKH}
\end{figure}

\vspace{-0.5em}
\begin{figure}[htb!]
	\centering
    \subfigure[$<B_p^2>$]{
	\includegraphics[width=0.4\linewidth]{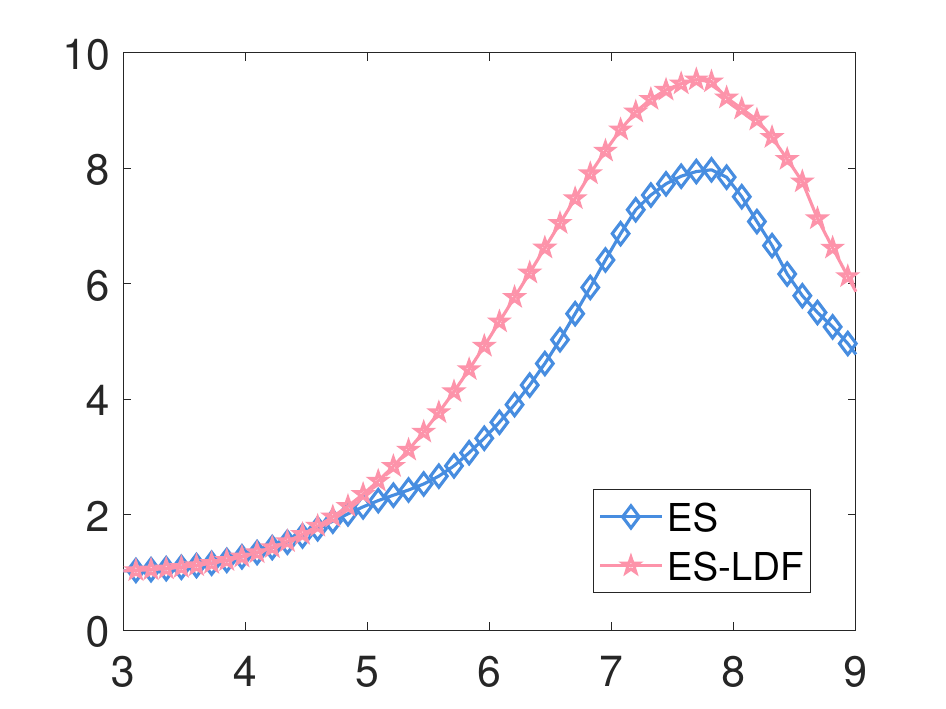}}
	\subfigure[$\Delta u_y$]{
	\includegraphics[width=0.4\linewidth]{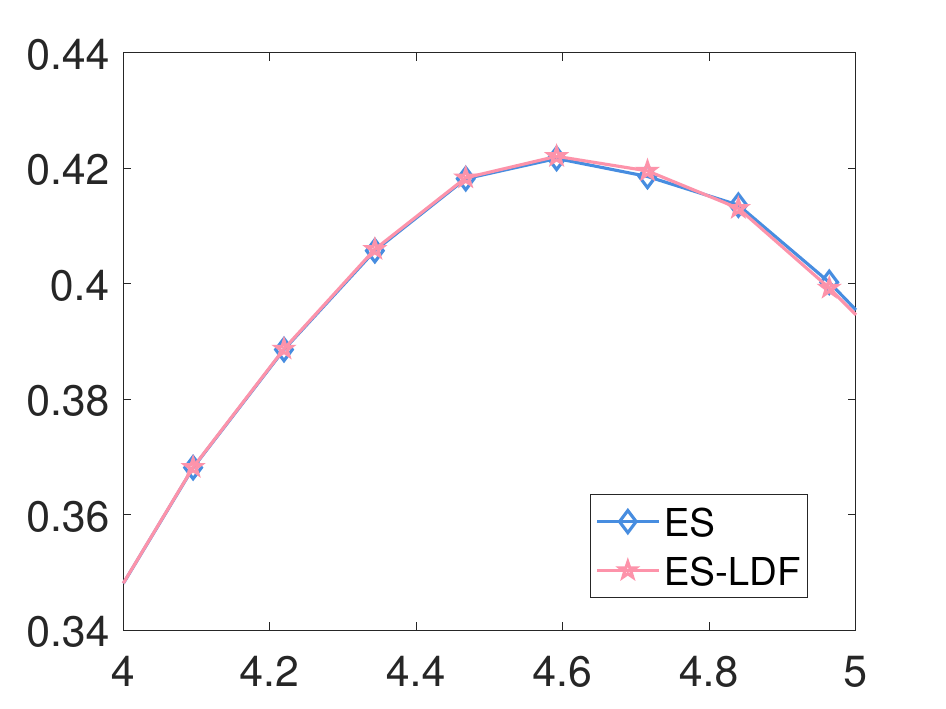}}
	\caption{Example \ref{ex:KH}: Kelvin--Helmholtz instability. The evolution of $<B_p^2>$ and $\Delta u_y$ with time. }
    \label{figKH2}
\end{figure}

\end{Ex}

\begin{Ex}[Cloudshock interaction]\label{ex:Cloudshock} This test \cite{rossmanith2006unstaggered} simulates the dynamical process of a strong shock wave impacting a stationary, high-density bubble. After the shock passes the bubble, very complex structures appear in the computational domain. In numerical simulations, this test is primarily used to verify the solver's ability to capture strong shocks, bow shocks, and transmitted shocks under extreme deformation conditions. Moreover, those structures around the bubble are susceptible to numerical dissipation, and low-dissipation schemes are advantageous for obtaining sharper structures. We use the setup in \cite{christlieb2014finite}. For this problem, the PP limiter is used. 

We simulate this problem using $320 \times 320$ meshes until $T=0.06$. For this example, the PP limiter allows all six schemes to remain stable, and the results are similar. The density results for the ES-LDF scheme are shown in Fig. \ref{figCloudshock}. It can be seen that the complex structures are well resolved with the proposed ES-LDF scheme, demonstrating the low numerical dissipation of our scheme. The solution agrees well with the results in the literature \cite{christlieb2014finite, wu2018provably, liu2025structureMHD, liu2025globally}. 

\begin{figure}[htb!]
	\centering
    \subfigure[$\ln\rho$.]{
		\includegraphics[width=0.31\linewidth]{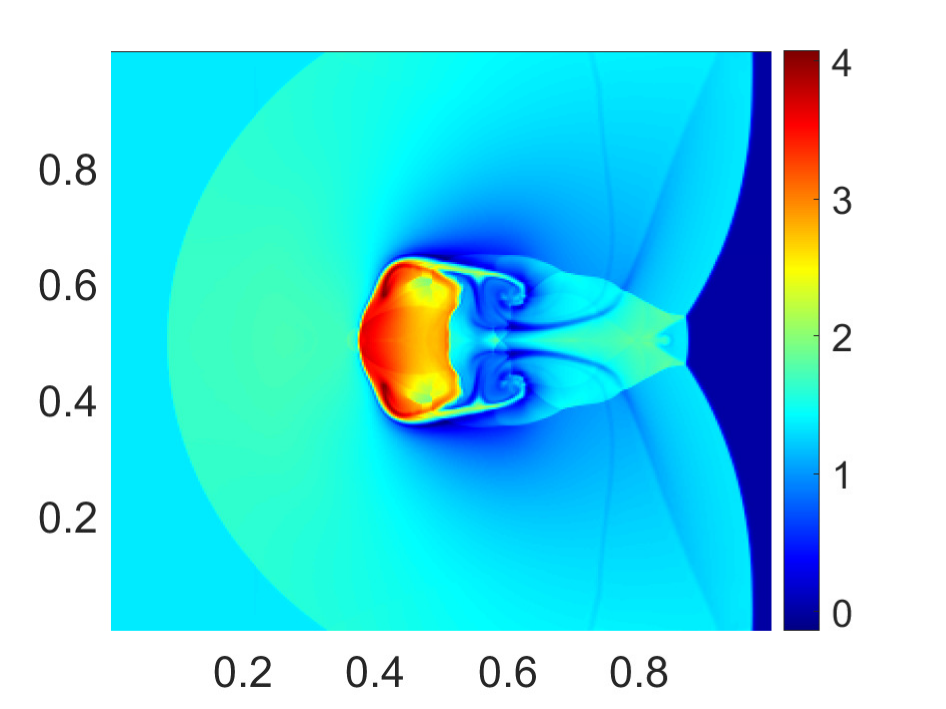}}    
	\subfigure[$p$.]{
	\includegraphics[width=0.31\linewidth]{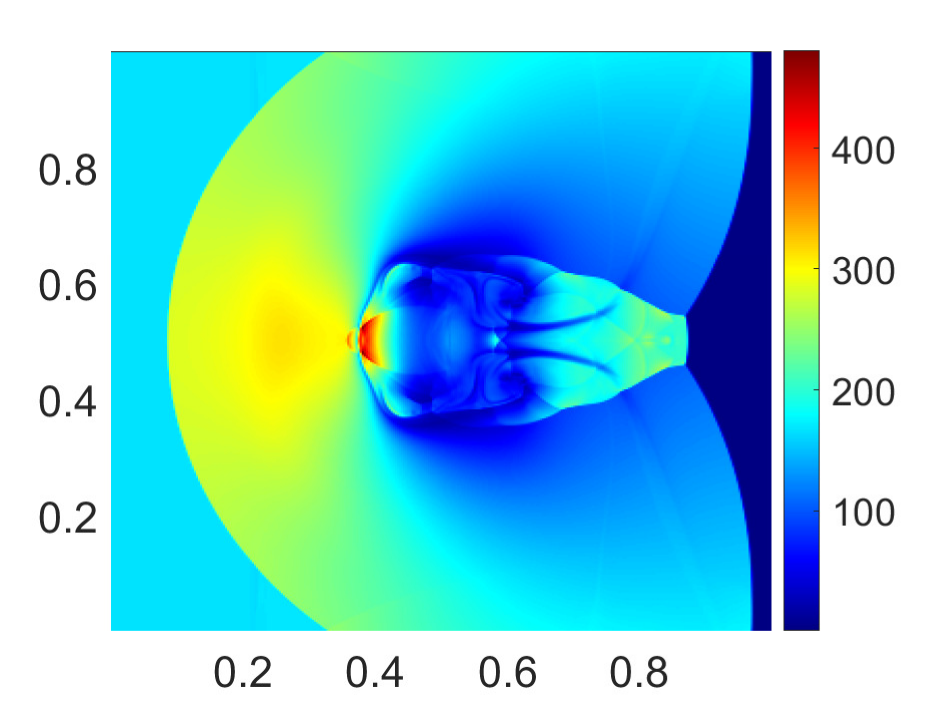}}
    \subfigure[$\left\|\mathbf B\right\|$.]{
		\includegraphics[width=0.31\linewidth]{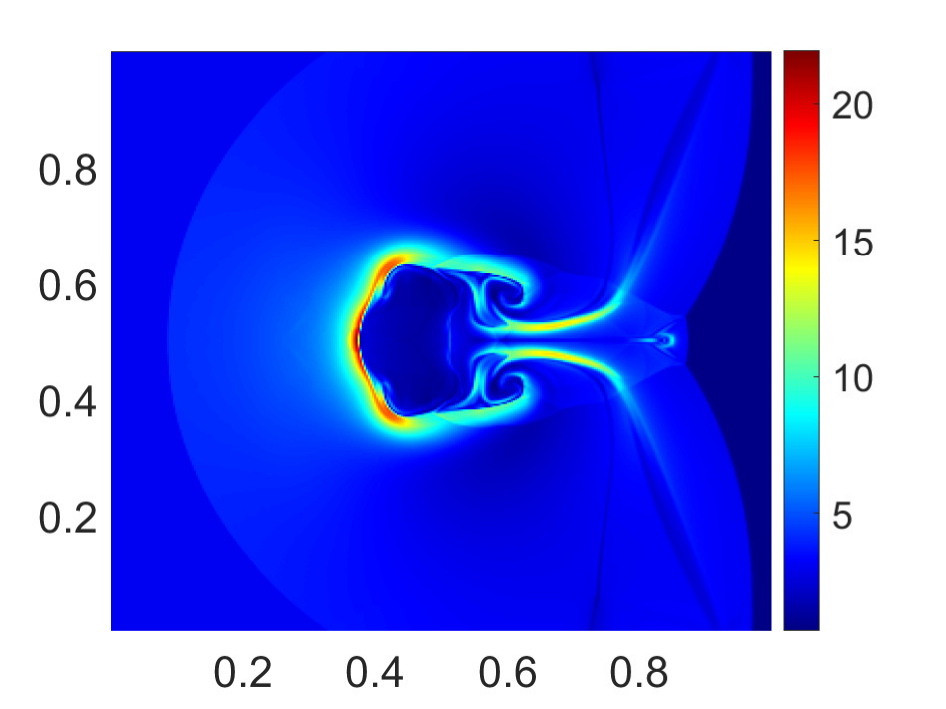}}
	\caption{Example \ref{ex:Cloudshock}: Cloudshock interaction. $N_x\times N_y=320\times 320$, $T=0.06$.  }
    \label{figCloudshock} 
\end{figure}

\end{Ex}

\begin{Ex}[Extreme MHD jet]\label{ex:jet} This test simulates astrophysical jet phenomena under extreme conditions, which was first introduced by Balsara \cite{balsara2012self} and later modified by Wu \textit{et al.} in \cite{wu2018provably, liu2025structureMHD, yan2025provably} to create more challenging configurations, serving as an ultimate stress test to evaluate the robustness of the solver. Here we test the most extreme case in \cite{yan2025provably} with $B_a=\sqrt{20000}$ and $u_{\mathrm{jet}}=1000000$. For this example, the PP limiter is employed. We particularly emphasize that $\gamma=1.4$ for this example. 

We only simulate the right half-domain $x>0$ and employ a reflective boundary condition at $x=0$. The part $x<0$ is obtained by symmetry. The test is run on $200 \times 600$ meshes up to $T=0.0000015$. We note here that the non-ES schemes (Base, LDF, SG, SG-LDF) will all blow up for this challenging example even if the PP limiter is employed.  In Fig. \ref{figjet}, we show the results of the ES-LDF scheme, and those of the ES scheme are similar. With the ES and PP limiters, the computation remains completely stable, accurately tracking the jet propagation and the extreme bow shock front, demonstrating a significant advantage of maintaining fully-discrete entropy stability. 

\begin{figure}[htb!]
	\centering
	\subfigure[$\ln\rho$.]{
		\includegraphics[width=0.31\linewidth]{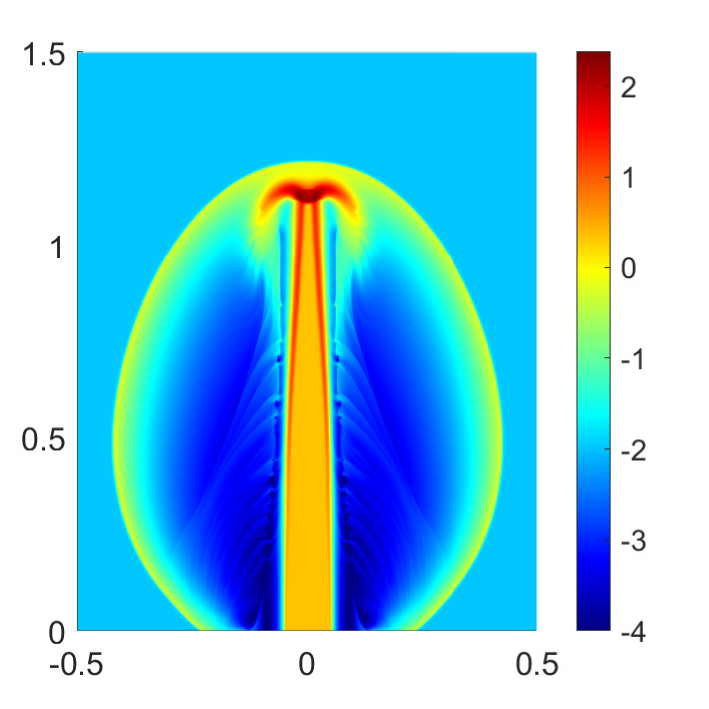}}
	\subfigure[$\left\|\mathbf u\right\|$.]{
	\includegraphics[width=0.31\linewidth]{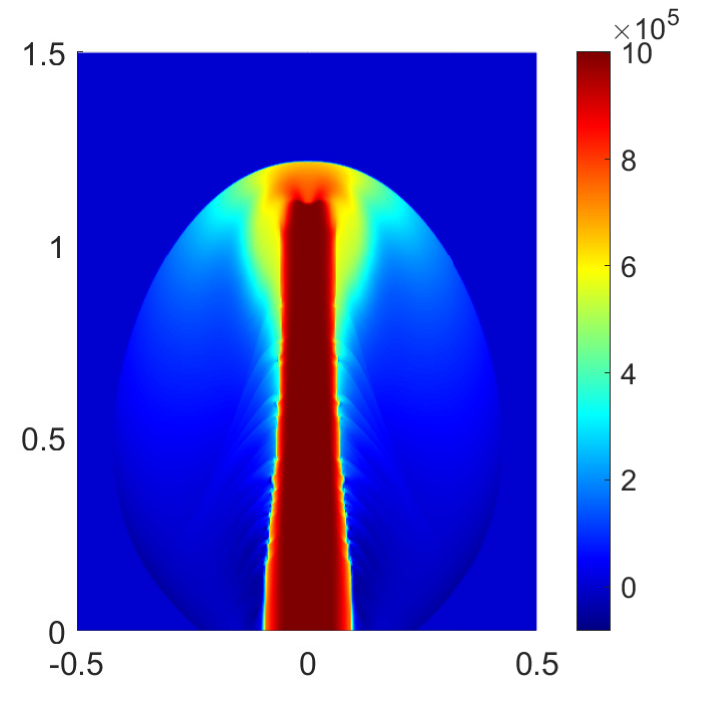}}
    \subfigure[$\ln (\left\|\mathbf B\right\|^2)$.]{
	\includegraphics[width=0.31\linewidth]{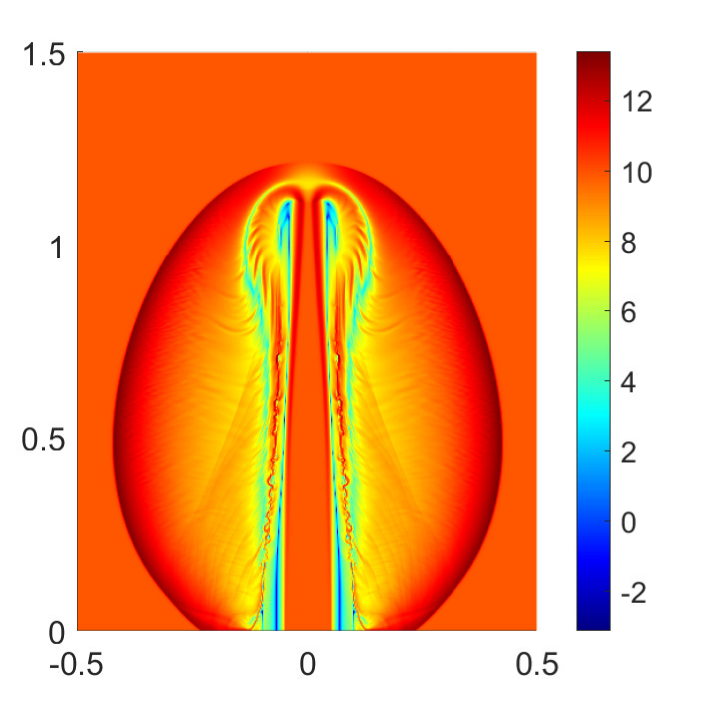}}
	\caption{Example \ref{ex:jet}: Extreme MHD jet. $N_x\times  N_y=200\times 600$, $T=0.0000015$.  }
    \label{figjet}
\end{figure}

\end{Ex}

\section{Concluding remarks} \label{sec7}

In this paper, we develop and analyze a class of high-order fully-discrete ES explicit DG schemes for ideal MHD equations that are also compatible with the LDF space. Based on Godunov's symmetrizable form, our approach addresses the highly non-trivial challenge of strictly preserving fully-discrete entropy stability for non-conservative MHD systems while maintaining high-order spatial and temporal accuracy. 

The main theoretical contribution lies in the proposal of a novel generalized-path-decomposition framework. By innovatively interpreting the interior volume integral of the non-conservative source term as an exact integral along a polynomial path constructed by the DG solution itself, we successfully establish the weak entropy inequality for cell average update schemes. Coupled with the ES limiter, the proposed explicit scheme strictly enforces the genuine cell entropy inequality and global entropy stability, which are subsequently used to obtain a Lax--Wendroff-type theorem guaranteeing that the solution limit satisfies the entropy condition. Furthermore, we demonstrate the broad applicability of this framework by showing that several existing classical solvers, which can be viewed as methods with low-order approximations of path integrals, share similar theoretical properties. Finally, we combine the proposed scheme with LDF methods to enforce zero divergence within each cell. Extensive numerical tests show that the proposed method not only enjoys low numerical dissipation for resolving complex structures but also exhibits strong robustness. 

Our future work will focus on extending the current ES framework to globally divergence-free (GDF) methods, aiming to construct fully-discrete ES, GDF, and strictly conservative DG schemes for ideal MHD equations.

\section*{Use of AI tools} The authors used ChatGPT (OpenAI, GPT-5.6 Sol) to assist with language editing, checking portions of the mathematical arguments, and revising their presentation. All AI-assisted content was independently verified by the authors, who assume responsibility for the final manuscript.

\bibliographystyle{siamplain}
\bibliography{references}

\end{document}